\documentclass[reqno]{amsart}
\usepackage[scale=0.75, centering, headheight=14pt]{geometry}
\usepackage[T1]{fontenc}
\usepackage{lmodern}
\usepackage[english]{babel}

\usepackage{tikz}
\usepackage{bbm}
\usepackage{amsmath,amssymb,amsfonts,amsthm}
\usepackage{mathtools,accents}
\usepackage{mathrsfs}
\usepackage{xfrac}
\usepackage{array} 
\usepackage{aliascnt}
\usepackage{booktabs} % for much better looking tables
\usepackage{array} % for better arrays (eg matrices) in maths

\usepackage{verbatim} % adds environment for commenting out blocks of text & for better verbatim
\usepackage{subfig} % make it possible to include more than one captioned figure/table in a single float

\usepackage{mathrsfs, dsfont}
\usepackage{amssymb}
\usepackage{amsthm}
\usepackage{amsmath,amsfonts,amssymb,esint}
\usepackage{graphics,color}
\usepackage{enumerate}
\usepackage{mathtools,centernot}

\usepackage{microtype}
\usepackage{paralist} % very flexible & customisable lists (eg. enumerate/itemize, etc.)
\usepackage{cases}
\usepackage[initials]{amsrefs}
\allowdisplaybreaks

\usepackage{braket}
\usepackage{bm}

\usepackage[citecolor=blue,colorlinks]{hyperref}
\addto\extrasenglish{}

\usepackage{enumerate}
\usepackage{xcolor}
\usepackage{aliascnt}
\numberwithin{equation}{section}

\makeatletter
\def\newaliasedtheorem#1[#2]#3{
  \newaliascnt{#1@alt}{#2}
  \newtheorem{#1}[#1@alt]{#3}
  \expandafter\newcommand\csname #1@altname\endcsname{#3}
}
\makeatother

\def\avint{\mathop{\mathchoice{\,\rlap{-}\!\!\int}
                              {\rlap{\raise.15em{\scriptstyle -}}\kern-.2em\int}
                              {\rlap{\raise.09em{\scriptscriptstyle -}}\!\int}
                              {\rlap{-}\!\int}}\nolimits}

\def\avint{\mathop{\,\rlap{-}\!\!\int}\nolimits}

\newcommand{\A}{\mathcal{A}}

\DeclareMathOperator{\id}{id_2}

\newcommand{\R}{\mathbb{R}}
\newcommand{\C}{\mathbb{C}}
\newcommand{\E}{\mathds{E}}
\newcommand{\EE}{\text{E}}

\newcommand{\N}{\mathbb{N}}

\newcommand{\eps}{\varepsilon}

\DeclareMathOperator{\cof}{cof}

\DeclareMathOperator{\Id}{Id}

\DeclareMathOperator{\dv}{div}

\DeclareMathOperator{\spt}{spt}

\DeclareMathOperator{\loc}{loc}

\DeclareMathOperator{\Lip}{Lip}

\DeclareMathOperator{\dist}{dist}

\theoremstyle{plain}

\newtheorem{Teo}{Theorem}[section]
\newtheorem{lemma}[Teo]{Lemma}
\newtheorem{prop}[Teo]{Proposition}

\newtheorem{Cor}[Teo]{Corollary}
\theoremstyle{definition}
\newtheorem{Def}[Teo]{Definition}
\theoremstyle{remark}
\newtheorem{rem}[Teo]{Remark}

\begin{document}

\title[prescribed vertical mean curvature system]{Regularity for the prescribed vertical mean curvature system in the plane}

\author[J. Hirsch]{Jonas Hirsch}
\address{Jonas Hirsch
\hfill\break Universit\"at Leipzig, Mathematisches Institut, Augustusplatz 10, 04109 Leipzig, Germany}
\email{hirsch.jonas@math.uni-leipzig.de}

\author[C. Mooney]{Connor Mooney}
\address{Connor Mooney
\hfill\break Department of Mathematics, UC Irvine}
\email{mooneycr@uci.edu}

\author[R. Tione]{Riccardo Tione}
\address{Riccardo Tione
\hfill\break Università degli Studi di Torino, dipartimento di Matematica “Giuseppe Peano”, Via Carlo Alberto 10, 10123 Turin, Italy}
\email{riccardo.tione@unito.it}

% ----------------------------------------------------------------
\begin{abstract}
We prove everywhere regularity for Lipschitz solutions to the prescribed vertical mean curvature system in the plane, where the codimension is arbitrary and the mean curvature is in $L^r, \, r > 2$.
\end{abstract}
% ----------------------------------------------------------------

\maketitle

\section{Introduction}

In this paper we study the regularity of Lipschitz maps $u = (u^1,\,\dots,\,u^n) : B_1 \subset \mathbb{R}^m \rightarrow \mathbb{R}^n$ that solve, in the sense of distributions, the system

\begin{equation}\label{PMC}
\frac{1}{\sqrt{\det g}}\partial_i(\sqrt{\det(g)}g^{ij}\partial_ju^l) = f^l, \quad l = 1,\dots,n.
\end{equation}
Here $g_{ij} = \delta_{ij} + \partial_iu \cdot \partial_ju,\, i,j = 1,\dots,m$, and $f = (f^1,\,\dots,\,f^n) \in L^r(B_1,\R^n),\, r \geq 1$.\ 

\medskip

If $u \in C^2$, system \eqref{PMC} says that the vertical projection of the mean curvature vector to the graph of $u$ (that is, the projection to the last $n$ variables in $\mathbb{R}^m \times \mathbb{R}^n$) is $f$. For this reason, we call \eqref{PMC} the \emph{prescribed vertical mean curvature system}. We also note that when $u \in C^2$, the horizontal projection $\tilde{f}$ of the mean curvature vector to the graph of $u$ is determined by $f$ and the fact that the mean curvature vector is normal to the graph of $u$,\ i.e. $\tilde{f} = -Du^Tf$. Equivalently, when $u \in C^2$ solves \eqref{PMC}, it also solves the system

\begin{equation}\label{PMCInner}
\frac{1}{\sqrt{\det g}}\partial_i(\sqrt{\det(g)}g^{ij}) = -f \cdot \partial_ju, \quad j = 1,\,\dots,\,m.
\end{equation}

\begin{rem}\label{variations}
More generally, whenever $v \in \Lip \cap W^{2,1}(B_1,\R^n)$ and $F$ is a smooth function on $\mathbb{R}^{n \times m}$, we have the implication 
\begin{equation}\label{eq:sys0}
\partial_i(F_{p^l_i}(Dv)) = H^l,\, l =1,\dots, n \quad \Longrightarrow \quad \partial_i(F(Dv)\delta_{ij} - F_{p^l_i}(Dv)v^l_j) = -H^lv^l_j, \, j =1,\dots, m.
\end{equation}
Here $F_{p^l_i}$ denotes the derivative of $F$ in the $p^l_i$ direction. The system on the left reflects that $v$ is a critical point of the energy 
\begin{equation}\label{ene}
\EE(v) = \int_{B_1} F(Dv) + (H,v)dx
\end{equation}
with respect to outer variations (i.e., $\EE(v + t\varphi) = \EE(v) + o(t)$ for all maps $\varphi \in C^{\infty}_c(B_1,\R^n)$), and the system on the right reflects that $v$ is a critical point of $\EE$ with respect to domain variations (i.e., $\EE(v(x + tV)) = \EE(v) + o(t)$ for all vector fields $V \in C^{\infty}_c(B_1,\R^m)$).\ If $D^2F > 0$ and $H \in L^2_{\loc}$, then Lipschitz solutions to the system on the left hand side of \eqref{eq:sys0} are in $W^{2,2}_{\loc}$ by standard regularity theory, and hence solve the system on the right hand side of \eqref{eq:sys0}. In this paper we are considering the area integrand $F(p) = \sqrt{\det(I + p^Tp)}$. This integrand is polyconvex (a convex function of the sub-determinants of $p$), and it is convex if and only if $\min\{m,\,n\} = 1$.
\end{rem}

If $u$ is Lipschitz and both systems \eqref{PMC} and \eqref{PMCInner} are assumed to hold with $f \in L^r_{\loc}$, then the graph of $u$ has mean curvature bounded in $L^r_{\loc}$ in the sense of varifolds and the regularity of $u$ is well-understood in all dimensions. Namely, provided $r > m$, we have $u \in W^{2,r}_{\loc}$ away from a closed singular set of Hausdorff dimension at most $m - 4$ \cite[Theorem 4.4]{DUG}. This bound on the size of the singular set is optimal even for $f = 0$, due to the Lawson-Osserman cone \cite[Theorem 7.1]{LOA}. Key tools in this setting that are not available when we only assume that \eqref{PMC} holds are the monotonicity formula and Allard's $\eps$-regularity theorem \cite{ALL}. 

\medskip

When $u$ is Lipschitz and we only assume that the system \eqref{PMC} holds, the situation is considerably more delicate, unless we are in the codimension one case $n = 1$.\ In that case solutions to \eqref{PMC} are minimizers of an energy of the form \eqref{ene}, and hence also solve \eqref{PMCInner}, and thus belong to $W^{2,r}_{\loc}(B_1,\R^n)$ when $r > m$, see \cite[Theorem 4.4 and Remarks]{DUG}. Celebrated examples of S. M\"{u}ller-V. \v{S}ver\'{a}k \cite{SMVS} and L. Sz{\'{e}}kelyhidi \cite{LSP} show that Lipschitz outer critical points of strongly quasiconvex, resp.\ polyconvex energies, can be extremely wild, e.g. nowhere $C^1$ (even when $m = n = 2$). Thus, positive regularity results for Lipschitz outer critical points of polyconvex energies must rely on additional structure.

\medskip

In their seminal work on the minimal surface system, H. B. Lawson and R. Osserman conjectured that Lipschitz solutions to \eqref{PMC} also solve \eqref{PMCInner} when $f = 0$, see \cite[Conjecture 2.1]{LOA}. In \cite{HMT} we proved the Lawson-Osserman conjecture in the planar case $m = 2$, for arbitrary codimension $n$. The purpose of this paper is to generalize this result to the case of nonzero right hand side. Our first main theorem is:

\begin{Teo}\label{thm:main}
If $u$ is a Lipschitz solution to (\ref{PMC}) in $B_1 \subset \mathbb{R}^2$ and $f \in L^r(B_1,\R^n)$ with $r > 2$, then $u \in W^{2,r}_{\loc}(B_{1},\R^n)$, hence $u \in C^{1,1-m/r}_{\loc}(B_1,\R^n)$ and solves \eqref{PMCInner}.
\end{Teo}

%\begin{rem}
%We in fact prove that $u \in W^{2,\,r}_{\loc}(B_1)$, thus (\ref{PMCInner}) holds, see Remark \ref{variations}.
%\end{rem}

Standard examples from Sobolev space theory show that we cannot expect that $u \in C^1$ when $r \leq 2$, even in codimension $n = 1$. For example, $u(x_1,\,x_2) \doteq x_1\sin(\log\log(1/|x|))$ is a Lipschitz function on a neighborhood of $0$ in $\mathbb{R}^2$ that solves (\ref{PMC}) with $f \in L^2$, but $u$ is not $C^1$.

\medskip

It would be extremely interesting to investigate what happens when $m \geq 3$. The counterexamples of \cite{SMVS,LSP} rely on showing that special sets of matrices known as $T_N$ configurations exist in the differential inclusion associated to the energy under investigation.\ It is natural to seek such configurations in the differential inclusions associated to the area functional.\ Through the same techniques used to show Theorem \ref{thm:main}, we prove the following rigidity statement for \emph{approximate solutions} to our PDE:
%This is obtained as corollary of the following rigidity statement for approximate solutions:

\begin{Teo}\label{thm:convun}
	Assume $(u_j)_j$ is a sequence of equibounded Lipschitz maps in $\Lip(B_1,\R^n)$, $B_1\subset \R^2$, whose weak-$*$ limit is the Lipschitz map $u$, and assume
	\[
		\dv(D\mathcal{A}(Du_j)) = \dv(F_j)
	\]
	with $F_j$ equibounded in $L^\infty(B_1)$ and converging strongly to $0$ in $L^1(B_1)$. Then, $(Du_j)_j$ converges strongly in $L^1$ and $u$ is a solution to $\dv(D\mathcal{A}(Du)) = 0$.
\end{Teo}

The absence of $T_N$ configurations in the inclusion set associated to critical points of the area functional when $m = 2$ is deduced in Remark \ref{rem:TN} as a corollary of this theorem.\ The same rigidity question remains open in higher dimensions.\ It would also be interesting to study the regularity of solutions to (\ref{PMC}) that are not assumed to be Lipschitz, but to lie in $W^{1,p}$ for some $p \geq 1$. We plan to investigate these directions in future work.

\medskip

We now describe our strategy for proving Theorem \ref{thm:main}. The starting point is the same as that in our proof for the case $f = 0$ in \cite{LOA}. Namely, by the use of isothermal coordinates (solving a Beltrami equation) we obtain a quasiconformal map $\psi$ such that the map $v = u \circ \psi$ solves the system
\[
\Delta v = \det D\psi\, h \circ \psi,\quad \text{ where $h \doteq \sqrt{\det g}f$.}
\]
When $f = 0$ we conclude that $v$ is harmonic; in particular, $v \in C^1$, which allowed us in \cite{LOA} to decompose $B_1$ into various regions where regularity of $u$ could be inferred for different reasons. However, when $f$ is nonzero, the above system for $v$ does not immediately give enough regularity to proceed in this manner. Instead, the key observation is that the system solved by $\psi$ implies that $(v,\,\psi)$ is a conformal map into $\mathbb{R}^{n+2}$. As a consequence, the Hopf differentials of $v$ and $\psi$ (see Section \ref{ProofMain} for the definition) sum to zero. After establishing an ellipticity property of the Hopf differential on quasiconformal matrices, we can use the regularity properties of $v$ from the system above to infer regularity properties of $\psi$ that are better than those we started with. We can then feed these back into the system above for $v$, bootstrapping until we get that both $v$ and $\psi$ are $C^{1,\,\alpha}$. The desired regularity for $u$ then follows at the points $\psi(y)$ such that $D\psi(y)$ is nonzero. In the final step of the proof we show that the set where $D\psi = 0$ is in fact empty, using ideas related to unique continuation originating in work of P. Hartman and A. Wintner \cite{HWW}.

\begin{rem}
In the last step of our proof we show the differential inequality $|\Delta \psi| \leq |Dv||h \circ \psi||D\psi|$. In particular, when $f = 0$, $\psi$ is harmonic, thus smooth with an isolated set of critical points. This gives a new proof, different from that in \cite{LOA}, of the regularity of $u$ in the case that $f = 0$.
\end{rem}

\begin{rem}
The Hopf differential also plays an important role in the theory of harmonic maps of Riemann surfaces, see \cite{HEL}.
\end{rem}

The paper is organized as follows. In Section \ref{PRELNOT} we establish notation and discuss some preliminary results about the structure of the area functional and quasiconformal maps. In Section \ref{ProofMain} we prove Theorem \ref{thm:main}, using a few results whose proofs are delayed to later sections.\ In Sections \ref{HopfEllipticity} and \ref{Bootstrapping} we prove those results, which include the ellipticity of the Hopf differential on quasiconformal matrices (Section \ref{HopfEllipticity}) and the bootstrapping argument alluded to above (Section \ref{Bootstrapping}).\ In Section \ref{sec:appsol} we use these techniques to show rigidity of approximate solutions, namely Theorem \ref{thm:convun}. Finally, in Appendix \ref{appa} we provide for the reader's convenience a proof of the Hartman-Wintner result with regularity hypotheses relaxed so as to apply in our setting.

\section*{Acknowledgments}
C. Mooney was supported by a Simons Research Fellowship and NSF grant DMS-2143668.\ R. Tione was funded by the Italian Ministry for University and Research (MUR), through FIS3 Starting Grant “GEMS” CUP: D53C25002540001 (Finanziata con il contributo del Ministero dell’Università e della ricerca ai sensi del D.D. n. 1802 del 21-11-2024 - BANDO FIS 3, Grant number FIS-2024-02219).

%%%%%

\section*{AI statement}

AI was not used in any part of the manuscript: the mathematical arguments were developed solely by the authors, who also wrote the entirety of the manuscript.

\section{Preliminaries and notation}\label{PRELNOT}
\subsection{Notation}
We denote by $X^T$ the transpose of $X \in \R^{n \times m}$. If $n = m$ and $X$ is invertible, we denote by $X^{-1}$ the inverse of $X$. We also write $\cof(X)$ for the cofactor matrix of $X$,  which satisfies
\begin{equation}
\label{eq:cof}
\cof(X) X = X \cof(X) = \det(X) \Id_m.
\end{equation}
Explicitly, in the case $n = 2$, if 
\[
X = \left(\begin{array}{cc}a & b\\ c & d\end{array}\right), \;\text{ then } \cof(X) = \left(\begin{array}{cc}d & -b\\ -c & a\end{array}\right).
\]
The standard Euclidean product of vectors $a,b \in \R^n$ is denoted by $(a,b)$, while the standard Euclidean product of matrices $A,B \in \R^{n\times m}$ is denoted by $\langle A,B \rangle$. The associated Euclidean norm is denoted both for vectors $a$ and matrices $A$ as $|a|$ and $|A|$.
\\
\\
We write $L^p(\Omega)$ if we consider $L^p$ scalar-valued functions, and $L^p(\Omega,\R^k)$ if we consider $\R^k$-valued functions. If we want to emphasize where the image of the map lies, say in $U \subset \R^k$, we will write $L^p(\Omega,U)$. We use the same convention for other functional spaces, such as Sobolev and H\"older spaces. When it is clear from the context, we will sometimes shorten the notation $\|h\|_{L^p(\Omega,\R^n)}$ to $\|h\|_p$.

\subsection{The area functional}
Here we recall some properties of the structure of the area functional. A more detailed explanation can be found in \cite[Section 2.2]{HMT}. As in that reference, we restrict our attention to the case in which the domain has dimension $m = 2$. As recalled in the introduction, the energy associated to the area integrand on graphs is
\[
\E(u) \doteq \int_{B_1}\A(Du)dx, \quad u \in \Lip(B_1,\R^n),
\]
where the area integrand on graphs $\mathcal{A}:\R^{n\times 2} \to \R$ reads as 
\[
\A(Z) \doteq \sqrt{\det(M(Z)^TM(Z))},\quad\text{ where } \quad
M(Z)\doteq
\left (
\begin{array}{c}
\id \\
 Z
\end{array}\right), \text{ and $Z \in \R^{n\times 2}$}.
\]
We say that a map $u \in \Lip(B_1,\R^n)$ has mean curvature $f \in L^p_{\loc}(B_1,\R^n)$, $p \ge 1$, if:
\begin{numcases}
\displaystyle\partial_i(\sqrt{\det g}g^{ij}\partial_ju^{l}) = \sqrt{\det g}f^l,& \quad $l$ = 1,\dots, $n$,\label{sys01}\\
\displaystyle\partial_i(\sqrt{\det g}g^{ij}) = -\sqrt{\det g}\partial_ju^if^i,& \quad $j$ = 1, 2,\label{sys02}
\end{numcases}
where $(g^{ij}) = (g_{ij})^{-1}$, $g_{ij} =\delta_{ij} + (D u^i,D u^j)$. System \eqref{sys01}-\eqref{sys02} is equivalent to asking that the graph of $u$, i.e. $\Gamma_u = \{(x,u(x)): x \in B_1\}$, has mean curvature bounded in $L^p$ in the sense of varifolds, see for instance \cite[Section 4.2]{SIM}, \cite[Section 6.3]{DLDPKT}.
If we introduce
\begin{equation}\label{Basg}
g(Z) \doteq \id + Z^TZ \text{ and } B(Z) \doteq \sqrt{\det g(Z)}g^{-1}(Z),
\end{equation}
we can rewrite system \eqref{sys01}-\eqref{sys02} in matrix notation:
\begin{numcases}
\displaystyle\dv(DuB(Du)) = \sqrt{\det g}f,\label{sys11}\\
\displaystyle\dv(B(Du)) = -\sqrt{\det g}Du^Tf,\label{sys12}
\end{numcases}
where the divergence is taken row-wise.\ Finally, if we make use of the algebraic identity
\begin{equation}\label{shortexpDA}
D\A(Z) = ZB(Z), \quad \forall Z \in \R^{n\times 2},
\end{equation}
we can finally rewrite \eqref{sys11} as 
\begin{equation}\label{eq:short}
	\dv(D\A(Du)) = h \doteq \sqrt{\det g}f.
\end{equation}
As $u \in \Lip(B_1,\R^n)$, $\det g$ is a measurable function bounded from above and below, hence $h \in L^p(B_1,\R^n)$ if and only if $f \in L^p(B_1,\R^n)$. In the sequel we will find convenient to use $h$ instead of $f$ when performing calculations and estimates.

\subsection{Quasiconformal mappings} We recall some useful properties of quasiconformal mappings.
\begin{Def}
Let $\Omega \subset \R^2$ be open. The map $\varphi\in W^{1,2}(\Omega,\R^2)$ is $(K-)$quasiregular if there exists $K \ge 1$ such that
\begin{equation}\label{Kqc1}
|D\varphi|^2(x) \le K\det(D\varphi(x)), \quad \text{for a.e. $x \in \Omega$.}
\end{equation}
We say that it is $(K-)$quasiconformal if, in addition, $\varphi$ is a homeomorphism\footnote{Note that every quasiregular map is continuous, compare Theorem \ref{thm:quas}\eqref{quasma4}.} between $\Omega$ and the open set $\varphi(\Omega)$.
\end{Def}

\begin{Teo}\label{thm:quas}
Let $\varphi \in W^{1,2}(\Omega_1,\Omega_2)$ be quasiregular and $\psi \in W^{1,2}(\Omega_2,\Omega_3)$ be quasiconformal with $\Omega_3 = \psi(\Omega_2)$. Then:
\begin{enumerate}
\item\label{quasma1} $\psi^{-1}: \Omega_3 \to \Omega_2$ is a $W^{1,2}(\Omega_3,\Omega_2)$ quasiconformal mapping (with the same quasiregularity constant $K$ of $\psi$, compare \eqref{Kqc1}), with gradient
\[
D(\psi^{-1}(y)) = (D\psi)^{-1}(\psi^{-1}(y)), \text{ for a.e. } y \in \Omega_3;
\]
\item\label{quasma2} the chain rule holds: for every $f \in W^{1,2}(\Omega_2)$, $f\circ \varphi \in W^{1,2}(\Omega_1)$ with
\[
D(f\circ\varphi)(x) = Df(\varphi(x))D\varphi(x), \text{ for a.e. } x \in \Omega_1;
\]
\item\label{quasma3} the composition $\psi \circ \varphi$ is quasiregular;
\item\label{quasma4} there exists $p = p(K) > 2$ such that $\varphi \in W^{1,p}_{\loc}(\Omega_1,\Omega_2)$, if $K$ is the quasiregularity constant of \eqref{Kqc1}, with bounds depending solely on the $W^{1,2}$ norm of $\varphi$;
\item\label{quasma5} $\det D\psi\neq 0$ a.e..
\end{enumerate}
\end{Teo}
Assertion \eqref{quasma1} can be found in \cite[Theorem 9.1]{BI} and \eqref{quasma2} is the content of \cite[Lemma 9.6]{BI}. \eqref{quasma3} is proved in \cite[Theorem 9.4]{BI}. Property \eqref{quasma4} is true in view of \cite[Theorem 5.1]{BI}, see also \cite{AST} for the sharp result. Finally, \eqref{quasma5} can be found in \cite[Theorem 7.2]{BI}.

\subsection{Complex derivatives}\label{s:cd}

In some points of the proof it is convenient to use holomorphic and anti-holomorphic coordinates. To recall them, let
\[
A = \left(\begin{array}{cc} a & b \\ c & d\end{array}\right).
\]
The conformal and anti-conformal parts are defined as follows, in matrix and complex notation:
\[
A_{co} = \frac{1}{2}\left(\begin{array}{cc} a + d & b-c \\ c-b & a + d\end{array}\right), \quad [A]_{\mathcal{H}} \doteq  \frac{1}{2}[(a + d) + i(c-b)]
\]
and
\[
A_{aco} = \frac{1}{2}\left(\begin{array}{cc} a - d & b+c \\ b+c & d - a\end{array}\right), \quad [A]_{\overline{\mathcal{H}}} \doteq  \frac{1}{2}[(a - d) + i(c+ b)].
\]
In particular, we have the fundamental identity
\[
Az = [A]_{\mathcal H}z + [A]_{\overline{\mathcal H}}\overline{z},
\]
where the left-hand side $Az$ is a matrix multiplication and the right-hand side is a sum of complex products of complex numbers.\ We have the following relations:
\begin{equation}\label{alg}
\det(A) = |[A]_{\mathcal{H}}|^2 - ||[A]_{\overline{\mathcal{H}}}|^2, \quad |A|^2 = 2|[A]_{\mathcal{H}}|^2 + 2|[A]_{\overline{\mathcal{H}}}|^2.
\end{equation}
Let $\Omega \subset \R^2$ be open. For a differentiable function $f:\Omega \to \R$, we can set
\begin{equation}\label{wirts}
\partial_zf = f_z \doteq \frac{1}{2}(\partial_1f - i \partial_2f) \text{ and } \partial_{\overline z}f = f_{\overline z} \doteq  \frac{1}{2}(\partial_1f + i \partial_2f) = \overline{f_z}.
\end{equation}
For $f = (f_1,\,f_2) : \Omega \to \R^2$, we define
\begin{equation}\label{wirtsint}
f_z \doteq \partial_zf_1 + i \partial_zf_2 \text{ and } f_{\overline z} \doteq \partial_{\overline z}f_1 + i \partial_{\overline z}f_2.
\end{equation}
More explicitly, in the latter case we can write:
\begin{equation}\label{wirt}
f_z = [Df]_{\mathcal{H}} = \frac{1}{2}[(\partial_1f_1 + \partial_2f_2) + i(\partial_1f_2 - \partial_2f_1)],\quad f_{\overline z} = [Df]_{\overline{\mathcal{H}}} = \frac{1}{2}[(\partial_1f_1 - \partial_2f_2) + i(\partial_1f_2 + \partial_2f_1)].
\end{equation}
This notation also yields another way to read \eqref{Kqc1}. First of all, assume that $Z \in \R^{2\times 2}$ is a quasiconformal matrix, i.e. that for some $K > 0$
\[
|Z|^2 \le K\det(Z).
\]
Then, due to \eqref{alg}:
\begin{align*}
|Z|^2 \le K\det(Z) \quad \Longleftrightarrow \quad 2| [Z]_{\mathcal{H}}|^2+2| [Z]_{\overline{\mathcal{H}}}|^2 \overset{\eqref{alg}}{\le} K(|[Z]_{\mathcal{H}}|^2-| [Z]_{\overline{\mathcal{H}}}|^2),
\end{align*}
and the latter is true if and only if there exists $\kappa = \kappa(K) \in [0,1)$ such that
\begin{equation}\label{qccomp}
|[Z]_{\overline{\mathcal{H}}}|^2 \le \kappa|[Z]_{\mathcal{H}}|^2.
\end{equation}
Using \eqref{qccomp} we see that $f \in W^{1,2}_{\loc}(\Omega,\R^2)$ is quasiregular if and only if there exists $\kappa \in [0,1)$ such that
\begin{equation}\label{fqccomp}
|f_{\overline z}|^2(z) \le \kappa|f_{z}|^2(z) \text{ for a.e. $z \in \Omega$}.
\end{equation}
We have the following result, which can be checked by hand and that we will use freely in the paper:
\begin{lemma}\label{teccomp}
Let $f,g: \Omega\subset \mathbb R^2 \to \mathbb R^2$ be sufficiently differentiable functions. Then,
\begin{enumerate}
\item $ f_{z\overline z} = f_{\overline z z} = \frac{1}{4}\Delta f$;
\item $(\overline f)_z = \overline{f_{\overline z}}$ and $(\overline f)_{\overline z} = \overline{f_{z}}$;
\item $(fg)_z = f_z g + fg_z$, $(fg)_{\overline z} = f_{\overline z} g + fg_{\overline z}$;
\end{enumerate}
\end{lemma}

\section{Proof of Theorem \ref{thm:main}}\label{ProofMain}
The proof is divided into steps. Following \eqref{eq:short}, we consider $u \in \Lip(B_1,\R^n)$ solving
\begin{equation}\label{eq:area}
\dv(D\A(Du)) = h, \quad \text{ for $h \in L^r(B_1,\R^n)$, $r > 2$.}
\end{equation}
 Our goal is to prove that $u \in W^{2,r}_{\loc}$, but it suffices to show that $u \in C^{1}(B_1,\R^n)$. Indeed, denoting $G_z(x) = G(x+z)-G(x)$ for any map $G$ and any $z$, we can consider $g \in W^{1,r}$ with $h = \dv(g)$ and write:
\[
\dv(g_z) = \dv(D\A(Du)(x+z)-D\A(Du)(x)) = \dv\left(\left[\int_0^1D^2\A((1-t)Du(x)+tDu(x+z))dt\right] Du_z\right).
\]
If $u \in C^1$, we can combine \cite[Theorems 4.4-7.3]{GM} to deduce the $W^{2,r}_{\loc}$ regularity of $u$, using that $\mathcal{A}$ fulfills a uniform Legendre-Hadamard condition on compact sets \cite[Lemma 6.7]{RT}.

\medskip

The starting point of the proof is precisely the same as the one of \cite[Proposition 3.3]{HMT}, i.e. as in \cite[Lemma 4.1]{HMT} we start by factorizing $u$ as:
\begin{equation}\label{eq:fact}
u = v\circ \varphi,
\end{equation}
where $v \in W^{1,2}(\varphi(B_1),\R^n)$ is shown to solve an auxiliary elliptic system \eqref{vf1} and $\varphi \in W^{1,2}(B_1,\R^2)$ is a quasiconformal map. The details are contained in Step 1. In Step 2 we will study the regularity of $v$ and of $\psi = \varphi^{-1}$. This exploits auxiliary results that will be proved in the forthcoming sections, Lemma \ref{lem:dett}, Theorem \ref{thm:regv} and Corollary \ref{thm:regp}. From the results of Step $2$, we will infer through \eqref{eq:fact} that $u$ is $C^{1,\alpha}$ outside of a singular set $\psi(S)$. In Step $3$ we will show how to remove this singular set with the help of Theorem \ref{thm:stronguq} and hence how to conclude the proof of the ongoing Theorem \ref{thm:main}.

\medskip

\fbox{Step 1: factorization of $u$.} We show that $u$ can be factorized as \eqref{eq:fact}. Let
\[
g \doteq \id + Du^TDu = \left(\begin{array}{cc} E& F\\ F & G\end{array}\right) \quad \text{ and }\quad  \mu(z) \doteq \frac{E-G+2iF}{E+G+2{\sqrt {EG-F^{2}}}}.
\]
$u \in \Lip(B_1,\R^n)$ implies $\id \le g \le \Lambda \id$ at a.e. point of $B_1$ in the sense of quadratic forms. This yields $|\mu|(z) \le k = k(\Lambda) < 1 \text{ for a.e. }z \in B_1$. Therefore, after extending $\mu$ to $0$ outside $B_1$, we can exploit \cite[Theorem 5.3.4]{AIM} to find a quasiconformal solution $\varphi \in W^{1,2}_{\loc}(\R^2,\R^2)$ to $\varphi_{\bar z} = \mu\varphi_z$ a.e. in $\R^2$.\ Passing to real notation, this is equivalent to:
\begin{equation}\label{gconf}
g = \rho D\varphi^TD\varphi, \quad \text{ with } \rho \doteq \frac{\sqrt{\det(g)}}{\det(D\varphi)}, \quad \text{ a.e. in $B_1$}.
\end{equation}
For details, see \cite[Theorem 10.1.1]{AIM}. Set $U \doteq \varphi(B_1)$, which is open since $\varphi$ is injective.\ As said, $v \in W^{1,2}(U,\R^n)$ defined by $u \circ \varphi^{-1}$ solves an auxiliary elliptic system: we write, for any $\eta \in C^\infty_c(U,\R^n)$,
\begin{equation}\label{vf}
\begin{split}
\int_{U}\langle Dv,D\eta\rangle dy &= \int_{U}\langle Du\circ \varphi^{-1}(D\varphi)^{-1}\circ \varphi^{-1},D\eta\rangle dy \\
&= \int_{B_1}\det(D\varphi)\langle Du(D\varphi)^{-1},D\eta\circ \varphi\rangle dx = \int_{B_1}\det(D\varphi)\langle Du(D\varphi)^{-1},D(\eta\circ \varphi)(D\varphi)^{-1}\rangle dx\\
& =\int_{B_1}\det(D\varphi)\langle Du(D\varphi)^{-1}(D\varphi)^{-T},D(\eta\circ \varphi)\rangle dx \overset{\eqref{gconf}}{=} \int_{B_1}\sqrt{\det(g)}\langle Dug^{-1},D(\eta\circ \varphi)\rangle dx\\
& \overset{\eqref{Basg}-\eqref{shortexpDA}}{=} \int_{B_1}\langle D\mathcal{A}(Du),D(\eta\circ \varphi)\rangle dx \overset{\eqref{eq:area}}{=} -\int_{B_1}(h,\eta\circ \varphi) dx.
\end{split}
\end{equation} 
In the last line we used the quasiconformality of $\varphi$ and Theorem \ref{thm:quas}\eqref{quasma2} to infer $\eta \circ \varphi \in W^{1,2}_0(B_1,\R^n)$, which is an admissible test map for the weak formulation of \eqref{eq:area}. Hence:
\begin{equation}\label{vf1}
\int_{U}\langle Dv,D\eta\rangle dy = -\int_{B_1}( h,\eta \circ \varphi) dx.
\end{equation}
Using the area formula, we see that the latter is just the weak formulation of
\[
	\Delta v = h\circ\psi\det(D\psi),
\]
a form that we will exploit in \eqref{vf111}.

\medskip

\fbox{Step 2: regularity of $v$ and $\psi$.} Bound the right-hand side of \eqref{vf1} through H\"older inequality:
\begin{align*}
\left|\int_{B_1}( h,\eta \circ \varphi) dx\right| \le \|h\|_{r}\|\eta\circ\varphi\|_{r'}.
\end{align*}
We let $\psi \doteq \varphi^{-1}$, which is quasiconformal due to Theorem \ref{thm:quas}\eqref{quasma1}.\ After the change of variables $x = \psi(y)$,
\begin{equation}\label{firprime}
\|\eta\circ\varphi\|_{r'}^{r'} = \int_{B_1}|\eta(\varphi(x))|^{r'}dx = \int_{U}|\eta(y)|^{r'}\det(D\psi)(y)dy.
\end{equation}
Let $p>2$ be such that $D\psi \in L^p_{\loc}(U,\R^2)$, see Theorem \ref{thm:quas}\eqref{quasma4}. We can then bound the previous integral  as:
\begin{equation}\label{seprime}
\int_{U}|\eta(y)|^{r'}\det(D\psi)(y)dy \le \left(\int_{U}|\eta(y)|^{r'\kappa}dy\right)^\frac{1}{\kappa}\left(\int_{\spt (\eta)}\det(D\psi)^\frac{p}{2}dy\right)^\frac{2}{p}.
\end{equation}
Here, $\kappa = \kappa(p)$ is the H\"older conjugate of $p/2$, namely
\[
\kappa(p) \doteq   \frac{p}{p-2}.
\]
Combining \eqref{vf1}-\eqref{firprime}-\eqref{seprime}, we find that
\[
\left|\int_{U}\langle Dv,D\eta\rangle dy\right| \le C\|\eta\|_{L^{r'\kappa}(U,\R^n)},
\]
for some $C$ that depends on $p$, $\psi$ and on $\dist(\spt(\eta), \partial U)$. Therefore, if $\beta$ is the H\"older conjugate of $r'\kappa(p)$:
\begin{equation}\label{eq:gainfi}
D\psi \in L^p_{\loc}(U,\R^2) \;\Longrightarrow\; \Delta v \in L_{\loc}^{\beta}(U,\R^n), \quad \text{for } \beta = (r'\kappa(p))' = \frac{r'p}{(r'-1)p + 2}.
\end{equation}
In addition to this information, we can find another relation between $v$ and $\psi$. Considering \eqref{gconf}, we write
\[
\id + Du^TDu = g = \rho D\varphi^TD\varphi.
\]
Using \eqref{eq:fact} and the chain rule we have, a.e. in $B_1$:
\[
\id + D\varphi^TDv^T\circ\varphi Dv\circ\varphi D\varphi = \rho D\varphi^TD\varphi.
\]
Inverting $D\varphi$ and using that $(D\varphi)^{-1}(x) = D\psi(\varphi(x))$ by Theorem \ref{thm:quas}\eqref{quasma1}, we can write
\[
D\psi^T(\varphi(x))D\psi(\varphi(x)) + Dv^T(\varphi(x)) Dv(\varphi(x)) = \rho \id.
\]
Finally we deduce that, for $a \doteq \rho \circ \psi(y)$:
\begin{equation}\label{eq:yUprel}
D\psi^T(y)D\psi(y) + Dv^T(y) Dv(y) = a(y) \id, \quad \text{for a.e. $y \in U$.}
\end{equation}
This tells us that the traceless parts of the symmetric matrices $D\psi^TD\psi$ and $Dv^TDv$ coincide a.e., up to a sign. Let us introduce\footnote{The Hopf differential arises when considering domain variations for the Dirichlet energy, in the sense that a map $u$ which is critical with respect to domain variations of the Dirichlet energy fulfills $\dv H(Du) = 0$.} the \emph{Hopf differential} operator $H: \R^{k\times 2} \to \R^{2\times 2}$, which yields, for all $Z \in \R^{k\times 2}$, the traceless part of $Z^TZ$:
\begin{equation}\label{defhopf}
H(Z) \doteq Z^TZ - \frac{|Z|^2}{2}\id.
\end{equation}
Here $k$ is any natural number and for us $k = 2$ or $k = n$ but, with a small abuse of notation, we will not denote $H$ differently for different $k$. Relation \eqref{eq:yUprel} then yields
\begin{equation}\label{eq:yU}
H(D\psi)(y) = -H(Dv)(y), \quad \text{for a.e. $y \in U$.}
\end{equation}
One last observation is that $u \in \Lip(B_1,\R^n)$ and the chain rule Theorem \ref{thm:quas}\eqref{quasma2} implies
\begin{equation}\label{eq:compara}
|Dv(y)| = |Du(\psi(y))D\psi(y)| \le C|D\psi(y)| \text{ a.e. in }U. 
\end{equation}
Consider then the system obtained combining \eqref{eq:gainfi}, \eqref{eq:yU} and \eqref{eq:compara}.
In Theorem \ref{thm:regv} we will show that this system implies that the map $v$ actually belongs to $W^{2,2+\eps}_{\loc}(U,\R^n)$, for some $\eps > 0$. Next, in Corollary \ref{thm:regp}, we will infer that also $\psi \in W^{2,2 + \eps}_{\loc}(U,\R^2)$ for a possibly smaller $\eps > 0$. This in particular yields that the gradients of $v$ and $\psi$ are $\alpha$-H\"older continuous. Now the classical inverse function theorem implies that $\varphi = \psi^{-1}$ belongs as well to $C^{1,\alpha}(E,\R^2)$ on the open set
\[
E = B_1\setminus \psi(S), \quad \text{ where } S\doteq \{y \in U: \det(D\psi(y)) = 0 \} = \{y \in U: |D\psi|(y) = 0\},
\]
where equality between definitions of $S$ holds since $\psi$ is quasiregular.\ Since $v \in C^{1,\alpha}(U,\R^n)$, \eqref{eq:fact} readily implies that $u \in C^{1,\alpha}(E,\R^n)$. To conclude that $u \in C^{1,\alpha}(B_1,\R^n)$ we need to show that $S = \emptyset$.
\\
\\
\fbox{Step 3: $S = \emptyset$ and conclusion of the proof.} To show that $S$ is empty, we need to use the claimed regularity of $v$ and $\psi$ to rewrite \eqref{vf1} and \eqref{eq:yU}. We notice in fact that \eqref{vf1} is equivalent to 
\begin{equation}\label{vf11}
\int_{U}\langle Dv,D\eta\rangle dy = -\int_{U}( h\circ\psi,\eta)\det(D\psi) dy
\end{equation}
after the usual change of variables $y = \varphi(x)$. Since $v \in W^{2,2+\eps}_{\loc}(U,\R^n)$, we can actually write \eqref{vf11} in its strong form a.e., denoting $F \doteq h\circ\psi$:
\begin{equation}\label{vf111}
\Delta v = F\det(D\psi).
\end{equation}
A simple change of variables and the fact that $D\psi \in L^\infty_{\loc}$ since $\psi \in W^{2,2+\eps}_{\loc}(U,\R^2)$ implies that
\begin{equation}\label{vfreg}
F\det(D\psi)^\frac{1}{2} \in L^{r}_{\loc}(U,\R^n).
\end{equation}
Since $v \in W^{2,2+\eps}_{\loc}(U,\R^n)$ and $\psi \in W^{2,2+\eps}_{\loc}(U,\R^2)$, we take the divergence of \eqref{eq:yU} pointwise a.e. to get:
\[
D\psi^T\Delta \psi = -Dv^T\Delta v \overset{\eqref{vf111}}{=} -Dv^TF\det(D\psi).
\]
As $D\psi$ is invertible almost everywhere by Theorem \ref{thm:quas}\eqref{quasma5}, we can rewrite the last equation as
\begin{equation}\label{stropsi}
\Delta \psi = -\cof^T(D\psi)Dv^TF, \quad\text{a.e. in }U.
\end{equation}
We can now bound:
\begin{equation}\label{stropsi}
|\Delta \psi| \le C|F||D\psi||Dv| \overset{\eqref{eq:compara}}{\le} C|D\psi|^2|F|,\quad \text{a.e. in }U.
\end{equation}
Setting $g \doteq |D\psi||F|$, we see from \eqref{vfreg} and the fact that $\psi$ is quasiregular that $g \in L^{r}_{\loc}(U)$. Thus, we rewrite the previous equation as:
\begin{equation}\label{stroinpsi}
|\Delta \psi|(y) \le Cg(y)|D\psi|(y), \text{ for a.e. $y$ in }U.
\end{equation}
Adapting an argument of \cite[Theorem 1-2]{HWW} from the case $g \in L^{\infty}_{\loc}(U)$ to the case $g \in L^r_{\loc}(U)$ and $r > 2$, we will show in Appendix \ref{appa} that any non-constant $W^{2,2+\eps}_{\loc}(U,\R^2)$ map $\psi$ fulfilling \eqref{stroinpsi} admits a precise asymptotic expansion at points of $S$. Assuming by contradiction that $S\neq \emptyset$ in our case, we can use this expansion to find contradict the quasiconformality of $\psi$. Let us show this.

\medskip

Assume without loss of generality $0 \in S$ and $\psi(0) = 0$, and up to scaling that $\psi$ is defined in $B_1$.
By Theorem \ref{thm:stronguq}, and using notation \eqref{oO}, we find that
\begin{equation}\label{expan}
\partial_z \psi_1 = a_1z^{k} + o(|z|^{k}) \text{ and } \partial_z\psi_2 = a_2z^{k} + o(|z|^{k}),
\end{equation}
for some $k \ge 1$, $a_1,a_2 \in \C$ with $|a_1|+|a_2| \neq 0$. The map $\psi$ is quasiregular, hence its gradient fulfills \eqref{fqccomp} for $\kappa \in (0,1)$. Making use of \eqref{wirts}, we can write at every point of $B_1$:
\begin{align*}
|\partial_{\overline z}\psi|^2 &= |\partial_{\overline z}\psi_1 + i\partial_{\overline z}\psi_2|^2 \overset{\eqref{wirts}}{=} |\partial_{z}\psi_1 - i\partial_{z}\psi_2|^2 \le \kappa |\partial_{z}\psi|^2 \overset{\eqref{wirts}}{=} \kappa|\partial_{z}\psi_1 + i\partial_{z}\psi_2|^2.
\end{align*}
Substituting \eqref{expan} in this inequality we readily find that
\begin{equation}\label{a1a2}
|a_1 - ia_2|^2\le \kappa |a_1 + ia_2|^2.
\end{equation}
Write, for a $R \in C^1(B_1,\R^2)$:
\begin{equation}\label{expanpsi}
\psi(z) = (a_1 + ia_2)\frac{z^{k + 1}}{k+1} + (\overline{a_1}+i\overline{a_2})\frac{\overline{z}^{k + 1}}{k + 1} + R(z).
\end{equation}
We claim that 
\begin{equation}\label{g}
R(z) = o(|z|^{k + 1}).
\end{equation}
To see this, we can observe that
\[
(a_1 + ia_2)z^k + o(|z|^k) \overset{\eqref{expan}}{=} \partial_z\psi_1(z) + i\partial_z\psi_2(z) \overset{\eqref{wirtsint}}{=} \partial_z \psi(z) \overset{\eqref{expanpsi}}{=} (a_1 + ia_2)z^k + \partial_zR(z).
\]
Thus we deduce $\partial_zR(z) = o(|z|^k)$. An analogous computation yields $\partial_{\overline z}R(z) = o(|z|^k)$. Due to \eqref{alg}, we then have $|DR|(z) = o(|z|^k)$. Since $\psi(0) = 0$, \eqref{expanpsi} yields $R(0) = 0$, and hence \eqref{g} holds. Now define 
\[
\tilde{\psi}(z) \doteq \frac{k + 1}{a_1 + ia_2}\psi(z) \quad \text{ and }\quad \lambda(z) \doteq \frac{\overline{a_1} + i\overline{a_2}}{a_1 + ia_2} + \frac{(k+1)R(z)}{(a_1+ia_2)\overline{z}^{k + 1}}.
\]
Notice that $a_1 + ia_2 \neq 0$ by \eqref{a1a2} and the fact that $|a_1| + |a_2| \neq 0$. Moreover, $\tilde{\psi}$ is still quasiconformal, and, assuming $|z| \le \rho$, we find $\delta > 0$ such that
\begin{equation}\label{rhode}
|\lambda|(z) \le 1-\delta
\end{equation}
by \eqref{a1a2} and \eqref{g}. Thus \eqref{expanpsi} can be rewritten as $\tilde{\psi}(z) = z^{k + 1} +\lambda(z)\overline{z}^{k + 1}$. Consider now the homotopy 
\[
H(t,z) \doteq z^{k + 1} +t\lambda(z)\overline{z}^{k + 1},\quad \text{ so that } H(0,z) = z^{k + 1} \text{ and } H(1,z) = \tilde{\psi}(z).
\]
Pick any regular value $y \in \tilde{\psi}(B_1)$ for $\tilde\psi$ with $0 < |y| \le \frac{\delta}{2}\rho^{k + 1}$.
%and
%\begin{equation}\label{regpsi}
%\det(D\tilde{\psi}(\tilde{\psi}^{-1}(y))) > 0.
%\end{equation}
Such $y$ exists since $\tilde{\psi}$ is a homeomorphism with $\tilde\psi(0) = 0$.\
% For such $y$, we have
%\[
%\deg(z^{k + 1},y,B_\rho) = k + 1,
%\]
%where $\deg$ is the Brouwer degree of a map, see \cite[Section 5]{MILN}. 
Our choice of $|y|$ implies that $y \notin H(t,\partial B_\rho)$ for any $t \in [0,1]$.\ Thus, the homotopy invariance of the Browder degree (see \cite[Section 5]{MILN} for reference) yields
\begin{equation}\label{degcon}
\deg(\tilde{\psi},y,B_\rho) =\deg(z^{k + 1},y,B_\rho) = k + 1 > 1,
\end{equation}
but, since $\tilde\psi$ is quasiconformal, $\deg(\tilde{\psi},y,B_{\rho})$ is either $1$ or 0, depending on whether $\tilde{\psi}^{-1}(y) \in B_\rho$ or not. In either case, we find a contradiction with \eqref{degcon} and we conclude the proof.

\begin{rem}
An alternate approach to Step 3 that avoids using the precise asymptotic expansion of $\psi$ is to use that $S$ consists of isolated points (also a consequence of Theorem \ref{thm:stronguq}). In particular, $u$ is $C^{1,\,\alpha}$ away from a set of isolated points. It follows that $u$ solves (\ref{sys01})-(\ref{sys02}) since isolated points are removable for this system, and one can then conclude using \cite[Theorem 4.4]{DUG}. However, it is interesting that one can show that $S$ is empty using the full strength of the expansion, so we have chosen to do so in the course of the proof.
\end{rem}

\section{An algebraic inequality}\label{HopfEllipticity}

In the proof of Theorems \ref{thm:regv}-\ref{thm:regp}, we need an inequality of algebraic nature which quantifies the ellipticity of the Hopf differential \eqref{defhopf} when restricted to quasiconformal matrices.\ This regularizing property of the Hopf differential is also at the core of \cite{HEL}. Although the proof is different and much shorter in this case, this section is analogous to \cite[Section 6]{HMT}, and is inspired by the work of V. \v Sver\'ak \cite[Theorem 3]{SAP} (see also \cite[Theorem 6.5]{RT} for further applications of these ideas).\ The main difference between \cite[Section 6]{HMT} and the situation at hand is that in \cite{HMT} we were interested in ellipticity properties of the inner variations for the area functional, while here we study similar properties of the Hopf differential, which arises as the matrix field associated to inner variations of the Dirichlet energy.\ We refer the reader to \cite[Lemma 4.6]{Guerra2026} and \cite[Section 9.1]{Iwaniec2013} for analogous computations.
\\
\\
Before stating the inequality, we rewrite the Hopf differential using complex notation:
\begin{equation}\label{eqH}
[H(X)]_{\mathcal{H}} = 0 \text{ and } [H(X)]_{\overline{\mathcal{H}}}= 2\overline{[X]_\mathcal{H}}[X]_{\overline{\mathcal{H}}}, \quad \text{ for any } X \in \R^{2\times 2}.
\end{equation}
The main result of this section is following:
\begin{lemma}\label{lem:dett}
There exists $C > 0 $ depending on $K > 0$ such that:
\begin{equation}\label{sptnull}
|X|^2 \le K \det X, |Y|^2 \le K\det Y\; \Longrightarrow\; |X-Y|^2 \le C\det(X-Y) +C\frac{|H(X)-H(Y)|^2}{\max\{|X|,|Y|\}^2}.
\end{equation}

\end{lemma}
\begin{proof}
The constant $C$ appearing in the next computations will always depend solely on $K$ and can change line-by-line. We start by noticing that, by symmetry, we can assume that $|Y|\le |X|$ and, in addition, that
\begin{equation}\label{YlessX} 
0 <|Y| \le |X|.
\end{equation}
Indeed if not $Y = 0$ and \eqref{sptnull} is true simply by quasiconformality of $X$, up to taking $C \ge K$.
\\
\\
We start by reducing the proof of \eqref{sptnull} to the proof of a similar inequality. To state it, we need to use the complex coordinates introduced in Section \ref{s:cd}. For simplicity, let us denote 
\[
a \doteq [X]_{\mathcal{H}},\;a' \doteq [X]_{\overline{\mathcal{H}}},\; b \doteq [Y]_{\mathcal{H}} ,\;b' \doteq  [Y]_{\overline{\mathcal{H}}}.
\] 
We claim that there exists $C = C(K)> 0$ and $\lambda = \lambda(K) \in [0,1)$ such that
\begin{equation}\label{totalc}
|a'-b'|^2 \le \lambda|a-b|^2 + C\frac{|\overline{a}a' - \overline{b}b'|^2}{|a|^2}.
\end{equation}
Let us first show how the claim implies the Lemma. Notice that
\begin{equation}\label{qccomp1}
\max\{|X|^2,|Y|^2\} = |X|^2 \overset{\eqref{alg}}{=} 2|a|^2 + 2|a'|^2 \overset{\eqref{qccomp}}{\le} C|a|^2.
\end{equation}
We fix $T = T(K)> 0$ large enough so that
\begin{equation}\label{T}
(1-T) +\lambda(1+T) < 0.
\end{equation}
We can do so since $\lambda = \lambda(K)$ appearing in \eqref{totalc} belongs to $[0,1)$. We can then compute:
\begin{align*}
2^{-1}|X-Y|^2 &\overset{\eqref{alg}}{=} |a - b|^2 +|a'-b'|^2 \\
&\overset{\eqref{alg}}{=} (1-T)|a-b|^2 +(1+T)|a'-b'|^2 + T\det(X-Y) \\
&\overset{\eqref{totalc}}{\le} (1-T)|a-b|^2 +\lambda(1+T)|a-b|^2 + T\det(X-Y) + (1+T)C\frac{|\overline{a}a' -\overline{b}b'|^2}{|a|^2} \\
&\overset{\eqref{T}}{\le} T\det(X-Y) + (1+T)C\frac{|\overline{a}a' - \overline{b}b'|^2}{|a|^2} \overset{\eqref{eqH}-\eqref{qccomp1}}{\le} T\det(X-Y) + C\frac{|H(X) - H(Y)|^2}{\max\{|X|,|Y|\}^2}
\end{align*}
This shows the validity of \eqref{sptnull} assuming \eqref{totalc}. Now, as a final step, we show \eqref{totalc}. Using the fact that $Y \neq 0$, \eqref{YlessX}, we can estimate for an $\eps = \eps(K) > 0$ to be fixed later,
\begin{equation}\label{a'b'}
\begin{split}
|a' - b'|^2  &= \left|\frac{a'\overline{a}}{\overline{a}} -\frac{b'\overline{b}}{\overline{b}}\right|^2 = \left|\frac{a'\overline{a} -b'\overline{b}}{\overline{a}} 
+ \left(\frac{1}{\overline{a}}-\frac{1}{\overline{b}}\right)b'\overline{b}\right|^2 \\
&\le \left(1+\frac{1}{\eps^2}\right)\left|\frac{a'\overline{a} -b'\overline{b}}{\overline{a}}\right|^2 +(1 +\eps^2)|a-b|^2\frac{|b'|^2}{|a|^2}.
\end{split}
\end{equation}
By \eqref{YlessX} and \eqref{alg}, we have that $|b'|^2 + |b|^2 \le |a'|^2 + |a|^2$. Using \eqref{qccomp} in this inequality, we find
\begin{equation}\label{kappa}
(1 + \kappa^{-1})|b'|^2 \le (1+\kappa)|a|^2\; \Rightarrow\; |b'|^2 \le \kappa |a|^2.
\end{equation}
Let us remark once again that $\kappa = \kappa(K) \in (0,1)$. We can then find $\eps = \eps(K) > 0$ such that $\lambda \doteq (1+ \eps^2)\kappa < 1$. For any such $\eps>0$, we see that \eqref{a'b'} implies \eqref{totalc} through \eqref{kappa} and we conclude the proof.
\end{proof}

\section{Auxiliary Results}\label{Bootstrapping}

\begin{Teo}\label{thm:regv}
Let $r > 2$ and let $v \in W^{1,2}(B_1,\R^n)$, $\psi \in W^{1,2}(B_1,\R^2)$, $f \in L^r(B_1,\R^n)$. Assume $\psi$ is quasiconformal and that the maps are linked through \eqref{eq:gainfi}, \eqref{eq:yU} and \eqref{eq:compara} for all $p > 2$, namely:
\begin{numcases}
\displaystyle D\psi \in L^p_{\loc}(B_1,\R^2) \Rightarrow \Delta v \in L_{\loc}^{\beta}(B_1,\R^n), \quad \text{for } \beta = \beta(p) = \frac{r'p}{(r'-1)p + 2}, \label{feB}\\
\displaystyle H(D\psi) = -H(Dv) \text{ a.e. in }B_1, \label{seB}\\
\displaystyle |Dv(y)| \le C|D\psi(y)| \text{ a.e. in }B_1.\label{thB}
\end{numcases}
Then, there exists $\eps > 0$ such that $v \in W^{2,2+\eps}_{\loc}(B_1,\R^n)$.
\end{Teo}
\begin{comment}
\begin{rem}
The previous result remains true even dropping assumption \eqref{thB}. However, this slightly shortens the proof and we have chosen to keep the assumption since this statement is anyway sufficient for our purpose, see \eqref{eq:compara}. However, we do not know whether Theorem \ref{thm:regp} is true without it.
\end{rem}
\end{comment}
\begin{proof}
The proof is based on an iterative argument.\ Namely, starting with a given regularity  $D\psi \in L^{p_i}_{\loc}(B_1,\R^2)$ for some $p_i > 2$, we can employ \eqref{feB} to deduce higher regularity of $v$. From \eqref{seB}, Lemma \ref{lem:dett} and a few additional observations we can then infer $D\psi \in L^{p_{i + 1}}_{\loc}(B_1,\R^2)$ with $p_{i + 1} > p_i$, compare \eqref{eq:p}. The goal is to show that, for sufficiently large $i$, $\beta(p_i) > 2$ and conclude the assertion through \eqref{feB} and standard Calder\'on-Zygmund theory. The iterative argument is hidden in the definition of $\sigma$ of \eqref{sup}, but it is the underlying idea of the proof. We start by describing one step of the iteration.

\medskip

Assume $D\psi \in L^p_{\loc}(B_1,\R^2)$ for some $p > 2$. The existence of a starting $p$ for which this is true is guaranteed by Theorem \ref{thm:quas}\eqref{quasma4}. We notice preliminarily that
\begin{equation}\label{bdend}
p > \frac{4}{2-r'}\quad \Longrightarrow\quad \beta(p) > 2.
\end{equation}
Thus we can, without loss of generality, assume that $2 < p <\frac{4}{2-r'}$.

\medskip

Consider $\eta \in C^{\infty}_c(B_1)$ with $\eta \equiv 1$ on $B_{\rho'}$ and $\spt(\eta)\subset B_\rho \subset B_1$.\ We employ \eqref{sptnull} with $X = D\psi(x+h)$ and $Y = D\psi(x)$, where $|h| <(1-\rho)/2$.\ For any map $f$, we set $f_h(x) \doteq f(x+h)-f(x)$.\ Since the determinant is a Null Lagrangian (see \cite[Theorem 2.3(ii)]{DMU}):
\begin{equation}\label{eq:routine0}
0 = \int_{B_1}\det(D(\eta(x)\psi_h(x)))dx = \int_{B_1}\eta^2\det(D\psi_h)dx + 2\int_{B_1}\eta\langle\psi_h\otimes D\eta, \cof^T(D\psi_h)\rangle dx.
\end{equation}
Using \eqref{sptnull} and routine estimates, see for instance \cite[Proposition 6.6]{RT}, this leads to
\begin{equation}\label{eq:routine}
\int_{B_1}\eta^2|D\psi_h|^2dx \le C\int_{B_1}|\eta||\psi_h||D\eta||D\psi_h| dx + C\int_{B_1}\eta^2\frac{|H(D\psi(x+h))-H(D\psi(x))|^2}{\max\{|D\psi(x+h)|,|D\psi(x)|\}^2}dx.
\end{equation}
In what follows, $C$ is a constant which depends on $p$ and on $\rho$ and is independent of $h$, and will change line-by-line. We observe that every component of $H(X)$ is a quadratic form, compare \eqref{defhopf}. Thus:
\begin{equation*}\label{eq:quadra}
\begin{split}
|H(D\psi(x+h))-H(D\psi(x))|^2& \overset{\eqref{seB}}{=} |H(Dv(x+h))-H(Dv(x))|^2 \\
&\le C\max\{|Dv|(x+h),|Dv|(x)\}^2|Dv_h|^2(x)\\
&\overset{\eqref{thB}}{\le} C\max\{|D\psi|(x+h),|D\psi|(x)\}^2|Dv_h|^2(x),
\end{split}
\end{equation*}
where the first inequality is true precisely due to the fact that $H$ is a quadratic form.\ This and Young's inequality allow us to estimate \eqref{eq:routine} as
%\[
%\int_{B_1}\eta^2|D\psi_h|^2dx \le 2\int_{B_1}|\eta||\psi_h||D\eta||D\psi_h| dx + C\int_{B_1}\eta^2|Dv_h|^2(x)dx,
%\]
%and, after Young's inequality, as
\begin{equation}\label{vpsi}
\int_{B_1}\eta^2|D\psi_h|^2dx \le C\int_{B_1}|D\eta|^2|\psi_h|^2 dx + C\int_{B_1}\eta^2|Dv_h|^2(x)dx.
\end{equation}
As $\psi \in W^{1,2}(B_1,\R^2)$, we can bound the first addendum simply by:
\begin{equation}\label{eq:routine1}
\int_{B_1}\eta^2|D\psi_h|^2dx \le C|h|^2 + C\int_{B_1}\eta^2|Dv_h|^2(x)dx.
\end{equation}
Estimating the second addendum requires more work.\ By \eqref{feB}, $v \in W_{\loc}^{2,\beta}(B_1,\R^n)$. Set $f^h(x) \doteq \eta(x)|Dv_h|(x)$ and choose
%We can interpolate its $L^2$ norm choosing as endpoints $L^{\beta}$ and $L^{\beta^*}$, where $\beta^*$ is the Sobolev exponent of $\beta < 2$. 
%Namely, choosing 
$\theta \in (0,1)$ as
\begin{equation}\label{theta}
	\theta \doteq \frac{4}{\beta}\left[1-\frac{1}{\beta}\right].
\end{equation}
We have $2= \theta \beta + (1-\theta)\beta^*$ and, using H\"older's inequality, we get:
\begin{equation}\label{fh}
\int_{B_1}|f^h|^2(x) dx \le \left(\int_{B_1}|f^h|^\beta(x) dx\right)^{\theta}\left(\int_{B_1}|f^h|^{\beta^*}(x) dx\right)^{1-\theta} \le C\left(\int_{B_1}|f^h|^\beta(x) dx\right)^{\theta},
\end{equation}
where here
\[
C = \sup_{h \in B_{\frac{1-\rho}{2}}}\left(\int_{B_1}|f^h|^{\beta^*}(x) dx\right)^{1-\theta} < + \infty.
\]
As $Dv \in W_{\loc}^{1,\beta}(B_1,\R^n)$, from \eqref{fh} we can then estimate \eqref{eq:routine1} as:
%\[
%\int_{B_1}|f^h|^2(x) dx \le C |h|^{\theta\beta}.
%\]
\begin{equation}\label{eq:routine2}
\int_{B_{\rho'}}|D\psi_h|^2dx \le \int_{B_1}\eta^2|D\psi_h|^2dx \le C |h|^{\theta\beta}.
\end{equation}
This means that $D\psi$ belongs to the local version of Besov space $B^{\frac{\theta \beta}{2},2}_\infty$, and we could already invoke embedding theorems concerning Besov spaces.\ Let us, instead, work with more classical fractional Sobolev spaces, which we denote by $H^t = W^{t,2}$.\ \eqref{eq:routine2} implies that $D\psi \in H^{t}(B_{\rho'},\R^2)$ for all $t < t_0 \doteq \frac{\beta\theta}{2}$. Indeed, 
%
%we have for a.e. $x \in B_{\rho'}$:
%\begin{align*}
%\int_{B_{\frac{1-\rho}{2}}(x)}\frac{|D\psi(x) - D\psi(y)|^2}{|x-y|^{2(1+t)}}dy = \int_{B_{\frac{1-\rho}{2}}}\frac{|D\psi(x) - D\psi(x+h)|^2}{|h|^{2(1+t)}}dh,
%\end{align*}
%where we have used the change of variables $y = x +h$. Thus,
\begin{align*}
\int_{B_{\rho'}\times B_{\rho'}}\frac{|D\psi(x) - D\psi(y)|^2}{|x-y|^{2(1+t)}}dxdy &= \int_{B_{\rho'}}\int_{B_{\rho'}}\frac{|D\psi(x) - D\psi(y)|^2}{|x-y|^{2(1+t)}}dydx \\
&= \int_{B_{\rho'}}\int_{B_{\frac{1-\rho}{2}}(x)}\frac{|D\psi(x) - D\psi(y)|^2}{|x-y|^{2(1+t)}}dydx \\
&\quad + \int_{B_{\rho'}}\int_{B_{\rho'}\setminus B_{\frac{1-\rho}{2}}(x)}\frac{|D\psi(x) - D\psi(y)|^2}{|x-y|^{2(1+t)}}dydx\\
& \le \int_{B_{\rho'}}\int_{B_{\frac{1-\rho}{2}}}\frac{|D\psi(x) - D\psi(x+h)|^2}{|h|^{2(1+t)}}dhdx + C\|D\psi\|_{2}^2.
\end{align*}
The last integral is finite due to \eqref{eq:routine2} and our choice of $t_0$. Thus, $D\psi \in H^{t}(B_{\rho'},\R^2)$, see \cite[Section 2]{HITCH}. From the fractional Sobolev embeddings \cite[Theorem 6.5]{HITCH} we find that then $D\psi \in L_{\loc}^\tau(B_{\rho'},\R^2)$ for $\tau = \frac{4}{2-2t} = \frac{2}{1-t}.$ In other words, if $\tau_0 \doteq \frac{2}{1-t_0}$ then, due to the arbitrariness of $\rho> 0$,
\[
D\psi \in L_{\loc}^{\tau_0 - \delta}(B_1,\R^2), \quad \forall \delta > 0.
\]
We can write more explicitly
\begin{equation}\label{tau0}
\frac{1}{\tau_0} =\frac{1}{2} -\frac{t_0}{2} = \frac{1}{2} -\frac{\beta\theta}{4} \overset{\eqref{theta}}{=}  \frac{1}{2} - \left[1-\frac{1}{\beta}\right] = \frac{1}{\beta} - \frac{1}{2} = \frac{1}{\beta^*}, \text{ i.e. } \tau_0 = \beta^* \overset{\eqref{feB}}{=} \frac{2r'p}{4 +(r'-2)p}.
\end{equation}
By hand, we can check that:
%\[
%2-\beta = 2 -  \frac{r'p}{(r'-1)p + 2} =  \frac{2(r'-1)p + 4-r'p}{(r'-1)p + 2} = \frac{r'p+ 2(2-p)}{(r'-1)p + 2}
%\]
%Hence
%\[
%\tau_0 = \beta^*  =\frac{2r'p}{r'p+ 2(2-p)}.
%\]
%Finally, we must study:
%\[
%\tau_0 - p = p\left(\frac{(2-r')(p-2)}{4 +(r'-2)p}\right).
%\]
%The numerator in the last expression is positive since $r > 2$ and $p>2$, and the denominator is positive since $p < \frac{4}{2-r'}$, compare \eqref{bdend}. 
%In other words,
\begin{equation}\label{eq:p}
\tau_0(p) \overset{\eqref{bdend}}{>} p,\; \forall p \in \left(2,\frac{4}{2-r'}\right) \;\; \text{ and } \;\; \tau_0(p) > \frac{4}{2-r'}, \; \forall   p \in \left(\frac{8}{4-(r')^2},\frac{4}{2-r'}\right).
\end{equation}
%We also have that
%\begin{equation}\label{eq:4}
%\tau_0(p) > \frac{4}{2-r'}, \quad \forall   p > \frac{8}{4-(r')^2}.
%\end{equation}
Let now
\begin{equation}\label{sup}
\sigma \doteq \sup\left\{s>1: D\psi \in L^s_{\loc}(B_1,\R^2)\right\}.
\end{equation}
As remarked above, $\sigma > 2$ due to Theorem \ref{thm:quas}\eqref{quasma4}. We claim 
\begin{equation}\label{eq:sigma}
\sigma >  \frac{8}{4-(r')^2}.
\end{equation}
If this is not the case, then
\[
\sigma \le \frac{8}{4-(r')^2} = \frac{2}{2 + r'}\frac{4}{2-r'} < \frac{4}{2-r'}.
\]
Thus one could repeat the reasoning above with $p < \sigma$ but very close to $\sigma$, and use the first inequality of \eqref{eq:p} to get a contradiction with the fact that $\sigma$ is a supremum.\ Since \eqref{eq:sigma} is verified, we can repeat once again the reasoning above with any $p \in\left(\frac{8}{4-(r')^2},\frac{4}{2-r'}\right)$ for which $D\psi \in L^p_{\loc}(B_1,\R^2)$ to obtain that $\sigma > \frac{4}{2-r'}$ by the second inquality of \eqref{eq:p} and conclude using \eqref{bdend}.
\end{proof}

\begin{Cor}\label{thm:regp}
Under the same assumptions of Theorem \ref{thm:regv}, it holds $\psi \in W^{2,2+\eps}_{\loc}(B_1,\R^2)$ for some $\eps>0$.
\end{Cor}
\begin{proof}
By Theorem \ref{thm:regv} we have $v \in W^{2,2+\eps}_{\loc}(B_1,\R^n)$ for some $\eps > 0$. We can then use \eqref{eq:routine1} to estimate
\[
\int_{B_1}\eta^2|D\psi_h|^2dx \overset{\eqref{eq:routine1}}{\le} C|h|^2 + C\int_{B_1}\eta^2|Dv_h|^2(x)dx \le C|h|^2,
\]
from which we obtain $\psi \in W^{2,2}_{\loc}(B_1,\R^2)$.\ To get the required higher integrability of $D^2\psi$, we employ Gehring's Lemma \cite[Proposition 5.1]{MGR} in a standard way, see \cite[Proposition 6.6]{RT} or \cite[Theorem 6.2]{HMT}.
% Let us recall the argument. We restart from \eqref{eq:routine0} with $\psi_h - a$ in place of $\psi_h$, for a constant $a \in \R^2$ to be chosen later. For the moment, consider the inequality for any $\eta \in C^\infty_c(B_1)$, and take any $h$ such that $h + \spt(\eta) \subset B_1$. The computations above can be repeated verbatim to finally get, instead of \eqref{vpsi}:
%\begin{equation}\label{vpsi1}
%\int_{B_1}\eta^2|D\psi_h|^2dx \le C\int_{B_1}|D\eta|^2|\psi_h-a|^2 dx + C\int_{B_1}\eta^2|Dv_h|^2(x)dx.
%\end{equation}
%Take any square $Q$ of side $r$ such that $4Q \subset B_{1}$. We denote by $\lambda Q$ the square centered at the center of $Q$ and with side $\lambda r$, for any $\lambda > 0$. We choose $\eta \in C_c^\infty(2Q)$, $\eta \equiv 1$ on $Q$, and $|D\eta(x)| \le Cr^{-1}, \forall x \in 2Q$. Combining this choice with the Poincarè-Sobolev inequality on cubes, we obtain from \eqref{vpsi1} and a suitable choice of $a$:
%\[
%\int_{Q}|D\psi_h|^2(x)dx \le \frac{C}{r^2}\left(\int_{2Q}|D\psi_h|(x)dx\right)^2 + C\int_{2Q}|Dv_h|^2(x)dx.
%\]
%Dividing by $|h|^2$ and letting $h \to 0$, we conclude that:
%\[
%\int_{Q}|D^2\psi|^2(x)dx \le \frac{C}{r^2}\left(\int_{2Q}|D^2\psi|(x)dx\right)^2 + C\int_{2Q}|D^2v|^2(x)dx,
%\]
%for a constant $C >0$ which only depends on $\dist(Q,\partial B_1)$, but is otherwise independent of $Q$. We invoke Gehring's Lemma as stated in \cite[Proposition 5.1]{MGR}. Since $D^2v \in L^{2+\eps}_{\loc}(B_1,\R^n)$, we obtain the existence of $\delta > 0$ such that $D^2\psi \in L^{2+\delta}_{\loc}(B_1,\R^2)$ and we conclude the proof.
\end{proof}

\section{Rigidity for approximate solutions}\label{sec:appsol}

In this section we show Theorem \ref{thm:convun}.\ Let us first introduce approximate solutions.\ Consider a rotation $J$ of 90 degrees, namely
\[
J\doteq \left(\begin{array}{cc}0 & -1 \\ 0 & 1\end{array}\right),
\] 
and define the inclusion set
\begin{equation}\label{diffinc}
K = \left\{Y\in\R^{(2n)\times 2}: Y=
\left(
\begin{array}{c}
	X\\
	D\mathcal{A}(X)J\\
\end{array}
\right)
\right\}.
\end{equation}
Notice that $w = (u,v): B_1\to \R^n\times \R^n$ solves $Dw \in K$ a.e.\ if and only if $\dv(D\A(Du)) = 0$. This is the standard way of turning \emph{exact solutions} to a PDE into \emph{exact solutions} to a differential inclusion (and vice versa). As said, we will consider \emph{approximate solutions}, i.e.\ we are interested in showing that given any sequence of equi-Lipschitz maps $(w_j)_j$, weakly-$*$ converging to $w \in \Lip(B_1,\R^{2n})$
\begin{equation}\label{P1}
\dist(Dw_j,K) \to 0 \text{ in $L^1(B_1)$}\; \Longrightarrow\; Dw_j \to Dw \text{ in }L^1(B_1,\R^{2n}).
\end{equation}
Equivalent rephrasings of this property are more convenient, compare to the statement of Theorem \ref{thm:convun}.\ The equivalence of these two statements is given by the next:

\begin{lemma}\label{equn}
	Let $K$ be defined in \eqref{diffinc}. Then, the following are equivalent:
	\begin{enumerate}
		\item \eqref{P1} holds; \label{appsol} 
		\item for every sequence $(u_j)_j$ weakly-$*$ converging to a Lipschitz map $u$ with
		\begin{equation}\label{approarea0}
			\dv(D\mathcal{A}(Du_j)) = \dv(F_j), \text{ for $F_j$ equibounded in $L^\infty(B_1)$, $F_j \to 0$ in $L^1(B_1)$,}
		\end{equation}
		 $(Du_j)_j$ converges strongly in $L^1$ towards $Du$; \label{pde}
%		\item\label{YM} for every gradient Young measure $\nu = (\nu_x)_{x \in B_1}$ with $\spt(\nu_x) \subset K$ for a.e. $x \in B_1$, it holds $\nu_x = \delta_{A(x)}$ for some $A(x) \in K$, for a.e. $x \in B_1$.
	\end{enumerate}
\end{lemma}
\begin{proof}
%	\eqref{P1} is readily seen to be equivalent to \eqref{YM} by Theorem \ref{FYM}, and in particular \eqref{strongc}.\ 
%	Let us now show the equivalence between $\eqref{appsol}$ and \eqref{pde}. 
	We start assuming $\eqref{pde}$. Consider $(w_j)_j$ as in \eqref{P1}, and define $w_j(x) = (u_j(x),v_j(x))$ with $u_j, v_j: B_1 \to \R^n$ equi-Lipschitz.\ Since $\dist(Dw_j,K) \to 0$ strongly in $L^1(B_1)$, we can set $F_j \doteq D\mathcal{A}(Du_j) +Dv_jJ$, which is equibounded in $L^\infty(B_1,\R^{n\times 2})$, and infer $F_j \to 0 \text{ strongly in $L^1(B_1,\R^{n\times 2})$}$.\ Thus,
	\[
	\dv(D\mathcal{A}(Du_j)) = \dv(D\mathcal{A}(Du_j) + Dv_jJ) = \dv(F_j),
	\]
	and hence by \eqref{approarea0} $(Du_j)_j$ converges strongly in $L^1$. As $Dv_j = -F_jJ + D\mathcal{A}(Du_j)J$, a sum of two strongly convergent sequences, also $(Dv_j)_j$ strongly converges in $L^1$, and hence $(Dw_j)_j$ does, as wanted. Conversely, assume that \eqref{appsol}, and thus \eqref{P1}, is true. Consider $(u_j)_j$ as in \eqref{pde}, and define $v_j$ through \eqref{approarea0}, namely
	\[
	Dv_j(x) = D\mathcal{A}(Du_j)(x)J - F_j(x)J, \quad \text{ for a.e. $x$ and for all $j$.}
	\]
	Setting $w_j = (u_j,v_j)$, this map is easily seen to fulfill the assumptions of property \eqref{P1}, and hence converges strongly due to our assumption \eqref{appsol}. It follows that \eqref{pde} holds.
\end{proof}

\begin{rem} \label{rem:TN}
	By combining Theorem \ref{thm:convun} and Lemma \ref{equn}, we obtain that every \emph{gradient Young measure} compactly supported in the inclusion set $K$ defined in \eqref{diffinc} must be trivial.\ In particular, the set $K$ does not contain any $T_N$ configuration, the \emph{laminates} used in \cite{SMVS,LSP} to construct irregular critical points. For the sake of brevity we do not provide the details, but we refer the reader to \cite[Section 3]{DMU} or to \cite[Section 4.1]{GT24} for the theory of Young measures, and to \cite{SMVS,LSP} for notions on laminates and $T_N$ configurations.
\end{rem}

\subsection{Proof of Theorem \ref{thm:convun}}\label{appsolrig}
Consider, as in the statement, a sequence $(u_j)_j$ weakly-$*$ converging to $u \in \Lip(B_1,\R^n)$ and satisfying
\begin{equation}\label{approarea}
	\dv(D\mathcal{A}(Du_j)) = \dv(F_j).
\end{equation}
It suffices to show that $Du_j$ converges strongly in $L^1(B_r)$, for any fixed $r < 1$. Thus, consider $r$ fixed.\ We start as in Step 1 of Theorem \ref{thm:main}, i.e. by factorizing $u_j$ as
	\begin{equation}\label{factn}
	u_j = v_j\circ \varphi_j.
	\end{equation}
	Indeed we let $g_j \doteq \id + Du_j^TDu_j$ in $B_1$ and $g_j \equiv \id$ in $B_1^c$.\ Since $g_j$ are equibounded, uniformly positive-definite symmetric tensors, we can find a $K$-quasiconformal solution $\varphi_j \in W^{1,2}_{\loc}(\R^2,\R^2)$ (for some $K$ independent of $j$) fulfilling
	\begin{equation}\label{gnconf}
		g_j = \rho_j D\varphi_j^TD\varphi_j, \text{ a.e. in $\R^2$,} \quad \text{ with } \rho_j \doteq \frac{\sqrt{\det(g_j)}}{\det(D\varphi_j)}.
	\end{equation}
	Compare to \eqref{gconf} and the details of that section. It is convenient to choose $\varphi_j$ to be the \emph{normalized} solution to the Beltrami equation associated to $g_j$, see \cite[Theorem 5.3.4]{AIM}, namely, using complex notation,
	\begin{equation}\label{reno}
	\varphi_j(0) = 0, \varphi_j(1) = 1, \text{ and }\varphi_j(\infty) = \infty \quad \forall j.
	\end{equation}
	Through this choice, $(\varphi_j)_j$ admits at least one weak $W^{1,2}_{\loc}(\R^2,\R^2)$ accumulation point $\varphi$, which is still $K$-quasiconformal and normalized \cite[Theorem 3.9.4]{AIM}. We assume without loss of generality that $\varphi$ is the weak $W^{1,2}_{\loc}(\R^2,\R^2)$ limit of $(\varphi_j)_j$.\ The same exact reasoning can be employed on $\psi_j \doteq \varphi_j^{-1}$, which are still $K$-quasiconformal by Theorem \ref{thm:quas}\eqref{quasma1} and fulfill \eqref{reno}.\ Assume, again without loss of generality, that $\psi$ is the only weak $W^{1,2}_{\loc}(\R^2,\R^2)$ quasiconformal limit of this sequence.\ The local uniform convergence provided by Theorem \ref{thm:quas}\eqref{quasma4} shows that $\psi = \varphi^{-1}$.\	Taking now any fixed $s<t \in (r,1)$, we consider the sets $U_j = \varphi_j(B_1)$, $U_t = \varphi(B_t)$, and $U_s = \varphi(B_s)$, which are open since $\varphi_j$ and $\varphi$ are homeomorphisms.\ A simple contradiction argument shows that: 
	\begin{equation}\label{Uset}
	\overline{\varphi_j(B_r)} \subset U_s \subset \overline{U_s} \subset U_t \subset \overline{U_t} \subset U_j, \quad \text{ for all } j \in \N \text{ sufficiently large}.
	\end{equation}
	For simplicity, and without loss of generality, we will assume that \eqref{Uset} holds for all $j$.
	
%	 We define, in addition, $\psi_j \doteq \varphi_j^{-1}$ and $v_j \doteq u_j \circ \psi_j$. 
	 
	 \medskip
	 
	We define $v_j \doteq u_j \circ \psi_j$. As in \eqref{eq:compara}, we notice that:
	\begin{equation}\label{eq:comparaj}
	|Dv_j(y)|\le C|D\psi_j(y)| \text{ a.e. in }U. 
 	\end{equation}
 	Throughout, $C$ is a constant that varies from line to line and which is independent of $j$. Since $(v_j)_j$ is also equibounded in $L^\infty(U_j,\R^2)$, \eqref{eq:comparaj} implies that $(v_j)_j$ is equibounded in $W^{1,2}(U_j,\R^2)$ and we can assume that it has a weak $W^{1,2}$ limit $v$ on ${U_t}$. We will show that 
 	\begin{equation}\label{strov}
 		\text{$v_j$ converges strongly in $W^{1,2}(U_s,\R^n)$ to $v$.}
 	\end{equation}
	 We compute the equation solved by $v_j$, similarly to what we did in \eqref{vf}. For any $\eta \in C^\infty_c(U_j,\R^2)$:
	\begin{equation}\label{vfn}
			\int_{U_j}\langle Dv_j,D\eta\rangle dy = \int_{B_1}\sqrt{\det(g_j)}\langle Du_jg_j^{-1},D(\eta\circ \varphi_j)\rangle dx.
	\end{equation}
	Now, combining \eqref{Basg} and \eqref{shortexpDA}, we obtain from \eqref{approarea} and \eqref{vfn} that:
	\begin{equation}\label{vfnn}
		\int_{U_j}\langle Dv_j,D\eta\rangle dy = \int_{B_1}\langle F_j,D(\eta\circ \varphi_j)\rangle dx.
	\end{equation}
	After a change of variable we finally get:
	\begin{equation}\label{vfnn}
		\int_{U_j}\langle Dv_j,D\eta\rangle dy = \int_{B_1}\langle F_j,D\eta(\varphi_j)D\varphi_j\rangle dx = \int_{U_j}\langle F_j\circ \psi_j,D\eta \cof(D\psi_j))\rangle dy.
	\end{equation}
	Set $G_j \doteq F_j\circ \psi_j\cof(D\psi_j)^T$.\ In the sense of distributions in $U_j$, $v_j$ solves
	\begin{equation}\label{pdevn}
	\Delta v_j = \dv(G_j).
	\end{equation}
	Using the fact that $\psi_j$ is $K$-quasiconformal, we can bound:
%	 compare Theorem \ref{thm:quas}\eqref{quasma1},
	\begin{align*}
	\int_{U_j}|G_j|^2dx &\le C\int_{U_j}|F_j\circ \psi_j|^2|D\psi_j|^2dx \le C\int_{U_j}|F_j\circ \psi_j|^2\det(D\psi_j)dx = C\int_{B_1}|F_j|^2dx.
	\end{align*}
	Hence, $\lim_{j\to \infty}\|G_j\|_{L^2(U_j)} = 0$. 
%	choosing $q = 2$, we obtain that
%	\begin{equation}\label{Gn}
%		 \lim_{n\to \infty}\|G_j\|_{L^2(U_j)} = 0.
%	\end{equation}
	From this, \eqref{Uset} and \eqref{pdevn} we obtain that $v$ is harmonic, and thus $v_j'\doteq v_j - v$ again solves \eqref{pdevn} in $U_t$.\ Testing this PDE with $\eta^2 v_j'$, for any $\eta \in  C^\infty_c(U_t)$, we get to:
	\begin{equation}\label{wn}
	\int_{U_t}|Dv'_j|^2\eta^2dx \le C \left( \int_{U_t}|v'_j|^2dx + \int_{U_j}|G_j|^2dx\right),
	\end{equation}
	which shows \eqref{strov}.

\medskip
	We now show that, up to a subsequence,
	\begin{equation}\label{strofi}
		\text{$\varphi_j$ converges strongly and pointwise a.e. in $W^{1,2}(B_s,\R^2)$ to $\varphi$.}
	\end{equation}
	Performing the same computations that brought us to \eqref{eq:yU}, we infer
	 \begin{equation}\label{eq:yUn}
		H(D\psi_j)(y) = -H(Dv_j)(y), \quad \text{for a.e. $y \in U_j$.}
	\end{equation}
	Using \eqref{sptnull}-\eqref{eq:comparaj}-\eqref{eq:yUn} and the quadratic structure of $H(\cdot)$, we have, for a.e. $y \in U_t$ and for all $j,m \in \N$, 
	\begin{equation*}
		\begin{split}
	|D\psi_j(y)-D\psi_m(y)|^2 \le C\det(D\psi_j(y)-D\psi_m(y)) +C|H(Dv_j(y))-H(Dv_m(y))|
	\end{split}
	\end{equation*}
	With computations analogous to the ones of \eqref{eq:routine0}, we obtain, for all $\eta \in C^\infty_c(U_t)$:
	\[
	\int_{U_t} |D\psi_j(y)-D\psi_m(y)|^2\eta^2 dy \le C(\eta)\left( \int_{U_t}|\psi_j(y)- \psi_m(y)|^2dy + \int_{U_t}|H(Dv_j(y))-H(Dv_m(y))|\eta^2dy\right).
	\]
	Using \eqref{strov}, we infer that $(D\psi_j)_j$ is a Cauchy sequence in $L^2(U_s,\R^2)$. Hence, $(\psi_j)_j$ converges to $\psi$ in $W^{1,2}(U_s,\R^2)$ and thus, up to a subsequence, $(D\psi_j)_j$ converges pointwise a.e. towards $D\psi$ on $U_s$. Now set
	\[
	f_j =\frac{|D\psi_j|^2}{\det(D\psi_j)},
	\]
	which is well-defined by Theorem \ref{thm:quas}\eqref{quasma5} and an equibounded sequence in $L^\infty(U_s)$ since $(\psi_j)_j$ is quasiconformal with uniform quasiconformality constant. We then have:
%	 By the fact that $(D\psi_j)_j$ converges pointwise a.e. to $D\psi$ and the dominated convergence theorem we can deduce
	\[
	\lim_{j}\int_{B_{s}}|D\varphi_j|^2dx = \lim_{j}\int_{U_s}|(D\psi_j)^{-1}|^2\det(D\psi_j)dy = \lim_{j}\int_{U_s}f_jdy = \int_{U_s}\frac{|D\psi|^2}{\det(D\psi)}dy = \int_{B_s}|D\varphi|^2dx,
	\]
%	
%	 We also deduce that $(\varphi)_j$ converges strongly towards $\varphi$ in $W^{1,2}(B_s,\R^2)$, indeed
%	\[
%	\int_{B_{s}}|D\varphi_j|^2(x)dx = \int_{U_s}|D\psi_j^{-1}|^2(y)\det(D\psi_j(y))dy = \int_{U_s}\frac{|D\psi_j|^2}{\det(D\psi_j)}(y)dy.
%	\]
%	The sequence $$f_j =\frac{|D\psi_j|^2}{\det(D\psi_j)}$$ is equibounded in $L^\infty(U_s)$ since $(\psi_j)_j$ is quasiconformal with a uniform quasiconformality constant. Due to the pointwise a.e. convergence of the gradients, we see that
%	\[
%	\lim_j\int_{U_s}\frac{|D\psi_j|^2}{\det(D\psi_j)}(y)dy = \int_{U_s}\frac{|D\psi|^2}{\det(D\psi)}(y)dy = \int_{B_s}|D\varphi|^2(x)dx.
%	\]
%	We obtain
%	\[
%	\lim_{n \to \infty} \int_{B_{s}}|D\varphi_j|^2(x)dx = \int_{B_s}|D\varphi|^2(x)dx,
%	\]
	from the dominated convergence theorem applied to the sequence $(f_j)_j$. This shows \eqref{strofi}.

\medskip

	We are finally ready to conclude. Let $\alpha \in (0,1)$ be sufficiently small. Recalling \eqref{strov} we have, for all $j$,
	\[
	\int_{B_r}|Dv_j(\varphi_j(x))|^\alpha dx = \int_{B_r}|Dv_j(y)|^\alpha \det(D\psi_j(y))dy \le C,
	\]
	due to H\"older inequality and Theorem \ref{thm:quas}\eqref{quasma4}, provided $\alpha$ is sufficiently small. Furthermore, we estimate
	\begin{align*}
	\int_{B_r}|Dv_j(\varphi_j)-Dv(\varphi)|^\alpha dx &\le \int_{B_r}|Dv_j(\varphi_j)-Dv(\varphi_j)|^\alpha dx + \int_{B_r}|Dv(\varphi_j)-Dv(\varphi)|^\alpha dx\\
	& = \int_{\varphi_j(B_r)}|Dv_j-Dv|^\alpha \det(D\psi_j)dy + \int_{B_r}|Dv(\varphi_j)-Dv(\varphi)|^\alpha dx.
	\end{align*}
	The first integral converges to zero due to \eqref{Uset}-\eqref{strov} and H\"older inequality, provided $\alpha$ is sufficiently small, and the second converges to zero since $v$ is harmonic and $(\varphi_j)_j$ converges uniformly to $\varphi$ on $B_r$. In particular, up to non-relabeled subsequences, $Dv_j(\varphi_j) \to Dv(\varphi)$ for a.e. $x \in B_r$. Through \eqref{factn}, we can write
	$Du_j(x) = Dv_j(\varphi_j(x))D\varphi_j(x)$, which converges a.e. in $B_r$ to $Dv(\varphi(x))D\varphi(x) \in L^2(B_r,\R^{n\times 2})$. Thus the dominated convergence theorem implies that $Du_j$ converges strongly in $L^2(B_r)$ to $Dv(\varphi(x))D\varphi(x)$.\ In particular, we must have that $Du(x) = Dv(\varphi(x))D\varphi(x)$ for a.e. $x \in B_r$.\ This reasoning shows that for every subsequence of $(u_j)$, there exists a further subsequence converging strongly in $W^{1,2}(B_r,\R^n)$ to $u$, from which we deduce the required strong convergence and hence the proof.

\appendix

\section{Hartman and Wintner's argument revisited}\label{appa}

In the following, we will use the classical notations $o(|z-x|^t), O(|z-x|^t)$ to denote $C^0$ functions defined on a neighborhood of the point $x$ such that
\begin{equation}\label{oO}
\lim_{z \to x}\frac{o(|z-x|^t)}{|z-x|^t} = 0 \text{ and } |O(|z-x|^t)| \le C|z-x|^t, \quad \text{for some $C > 0$}.
\end{equation}
Furthermore, for a map $w: B_1\to\R^n$, we will denote by $w_j$ its $j$-th component.\ Finally, for a continuous $h: \Gamma \to \C$, where $\Gamma$ is the image of a regular parametrization $\Gamma: [0,1]\to\C$, we set
\[
\int_{\Gamma}h d\gamma = \int_0^1 h(\Gamma(t))\Gamma'(t)dt \quad \text{ and }\quad \int_{\Gamma}|h| d|\gamma| = \int_\Gamma |h|(\Gamma(t))|\Gamma'(t)|dt.
\]
\begin{Teo}\label{thm:stronguq}
Let $n \ge 1$, $r>2$, $g \in L^{r}(B_1)$, $B_1\subset \R^2$, and $w \in W^{2,2+\gamma}_{\loc} \cap C^{1,\gamma}(B_1,\R^n)$ be a solution to
\begin{equation}\label{eq:intineq}
|\Delta w|(z) \le g(z)|Dw|(z) \quad \text{for a.e. }z \in B_1.
\end{equation}
Then either $Dw \equiv 0$ or the zeroes of $Dw$ are isolated. In this case, if $x \in B_1$ is such that $Dw(x) = 0$, there exists $k = k(x) \in \N$ and $a_j \in \C$ such that, for all $j =1,\dots, n$,
\[
\partial_z w_j(z) = a_j(z-x)^{k} + o(|z-x|^{k}),
\]
and $a_{j_0} \neq 0$ for some $j_0$.
\end{Teo}
%\begin{rem}
%By a simple bootstrap argument, in the previous theorem it would have been sufficient to assume that $w \in W^{2,1}_{\loc}(B_1,\R^n)$ is a solution to \eqref{eq:intineq} to deduce $w \in W^{2,2+\gamma}_{\loc} \cap C^{1,\gamma}(B_1,\R^n)$ for some $\gamma > 0$.
%%, but we chose to state it in this way to highlight the core of the statement, i.e. the properties of the zero set of $Dw$.
%\end{rem}
Theorem \ref{thm:stronguq} follows from adapting the arguments of \cite{HWW}. In particular, the proofs of Propositions \ref{filem}-\ref{selem} are essentially the same as the ones of \cite[Theorem 1-Theorem 2]{HWW} respectively.

\begin{prop}\label{filem}
Let $w$ fulfill the assumptions of Theorem \ref{thm:stronguq}. Then, for all $k \ge 1$ the following holds:
\begin{equation}\label{eq:secasew}
Dw(z) = o(|z-x|^{k-1})\quad \Longrightarrow \quad\lim_{z \to x}\frac{\partial_zw_j(z)}{(z-x)^{k}} \text{ exists for all $j$}.
\end{equation}
\end{prop}
\begin{proof}
%	[Proof of Proposition \ref{filem}]
Without loss of generality, we assume that $x = 0$. Consider $0<\rho< \frac{1}{2}$ to be chosen later. We take any $y \in B_{\rho}\setminus\{0\}$ and consider, for each $j$, $f_j(z) \doteq z^{-k}(z-y)^{-1}\partial_zw_j(z).$
%\[
%f_j(z) \doteq z^{-k}(z-y)^{-1}\partial_zw_j(z).
%\]
Let $0 <\sigma < |y|$ be so small that $B_\sigma(y) \cap B_\sigma(0) = \emptyset$ and set $\Omega_{\sigma} \doteq B_{\rho}\setminus (B_{\sigma}(y)\cup B_{\sigma}(0))$.\ By the divergence theorem and Lemma \ref{teccomp}, we see that
\begin{equation*}
\int_{\partial \Omega_{\sigma}} f_j(z)d\gamma(z) = 2i\int_{\Omega_{\sigma}} \partial_{\overline z}f_j(z)dz  = 2i\int_{\Omega_{\sigma}}z^{-k}(z-y)^{-1}\partial_{z\overline{z}}w_j(z)dz = \frac{i}{2}\int_{\Omega_{\sigma}}z^{-k}(z-y)^{-1}\Delta w_j(z)dz.
\end{equation*}
Using the structure of $\Omega_{\sigma}$ we get:
\begin{equation}\label{fiid}
\begin{split}
\int_{\partial B_{\rho}} f_j(z)d\gamma(z)- \int_{\partial B_{\sigma}(y)} f_j(z)d\gamma(z) - \int_{\partial B_{\sigma}} f_j(z)d\gamma(z)= \frac{i}{2}\int_{B_{\rho} \setminus (B_{\sigma}(y)\cup B_{\sigma})}z^{-k}(z-y)^{-1}\Delta w_j(z)dz.
\end{split}
\end{equation}
We have:
\[
\lim_{\sigma \to 0}\int_{\partial B_{\sigma}(y)} f_j(z)d\gamma(z) = 2\pi i y^{-k}\partial_zw_j(y).
\]
Furthermore by assumption we know that $\partial_zw_j(z) = o(|z|^{k - 1})$ and hence 
\[
\lim_{\sigma \to 0}\int_{\partial B_{\sigma}} f_j(z)d\gamma(z) = 0.
\]
We thus let $\sigma \to 0$ in \eqref{fiid} to get:
\begin{equation}\label{eqA}
\int_{\partial B_{\rho}} z^{-k}(z-y)^{-1}\partial_zw_j(z)d\gamma(z)-2\pi i y^{-k}\partial_zw_j(y) = \frac{i}{2}\int_{B_{\rho}}z^{-k}(z-y)^{-1}\Delta w_j(z)dz.
\end{equation}
The last passage to the limit in the right-hand side is justified by the fact that, for all $y \in \R^2\setminus\{0\}$,
\[
z^{-k}(z-y)^{-1}\Delta w_j(z) \in L^1_{\loc}(B_1).
\]
To see this, notice that the only non-trivial pole is $0$, but, since $Dw(z) = o(|z|^{k-1})$, we have for a.e. $z \in B_\rho$,
\begin{equation}\label{boundw}
|z^{-k}(z-y)^{-1}\Delta w_j(z)| \overset{\eqref{eq:intineq}}{\le} |z|^{-k}|z-y|^{-1}g(z)|D w(z)| \le C'|z|^{-1}|z-y|^{-1}g(z).
\end{equation}
The constant $C'$ depends on $k$.\ Instead, $C > 0$ will denote a universal constant that can change line by line but that does not depend on $k$. Resuming our reasoning, the last function in \eqref{boundw} belongs to $L_{\loc}^1$ since $r > 2$. We need to estimate the right-hand side of \eqref{eqA}. To do so, we first estimate \eqref{eqA} using \eqref{eq:intineq}:
\[
| y^{-k}\partial_zw_j(y)| \le \int_{\partial B_{\rho}} |z|^{-k}|z-y|^{-1}|\partial_zw_j(z)|d|\gamma|(z) + C\int_{B_{\rho}}|z|^{-k}|z-y|^{-1}g(z)|Dw(z)|dz.
\]
Summing over $j$, we get to
\begin{equation}\label{eqzy}
| y|^{-k}|Dw(y)| \le C\int_{\partial B_{\rho}} |z|^{-k}|z-y|^{-1}|Dw(z)|d|\gamma|(z) + C\int_{B_{\rho}}|z|^{-k}|z-y|^{-1}g(z)|Dw(z)|dz.
\end{equation}
For any $z_0 \in B_\rho\setminus\{0\}$, multiply by $g(y)|y - z_0|^{-1}$ this expression and integrate over $B_{\rho}$ to get
\begin{equation}\label{eqzya}
\begin{split}
\int_{B_\rho}| y|^{-k}|Dw(y)|g(y)|y-z_0|^{-1}dy & \le C\int_{\partial B_{\rho}} |z|^{-k}\left[\int_{B_\rho}g(y)|y-z_0|^{-1}|z-y|^{-1}dy\right] |D w(z)|d|\gamma|(z) \\
&+ C\int_{B_{\rho}}|z|^{-k}\left[\int_{B_\rho}g(y)|y-z_0|^{-1}|z-y|^{-1}dy\right] g(z)|Dw(z)|dz.
\end{split}
\end{equation}
Notice that $|(z-y)(y-z_0)|^{-1} = |z-z_0|^{-1}|(z-y)^{-1} + (y-z_0)^{-1}|$, therefore:
\begin{equation}\label{eq:prelll}
\begin{split}
\int_{B_\rho}g(y)|(z-y)(y-z_0)|^{-1} dy &=\int_{B_\rho}g(y)|z-z_0|^{-1}|(z-y)^{-1} + (y-z_0)^{-1}| dy\\
& \le |z-z_0|^{-1}\int_{B_\rho}g(y)[|z-y|^{-1} + |y-z_0|^{-1}] dy.
\end{split}
\end{equation}
We observe that, due to H\"older's inequality and the fact that $r > 2$, for all $b \in \overline{B_\rho}$
\[
\int_{B_\rho}g(y)|y-b|^{-1}dy \le \|g\|_{L^r}\left(\int_{B_{\rho}}|y-b|^{-r'}dy\right)^{\frac{1}{r'}} \le C\left(\int_{B_{2\rho}(b)}|y-b|^{-r'}dy\right)^{\frac{1}{r'}}\le C\rho^{\frac{2}{r'}-1}.
\]
Thus, we can continue \eqref{eq:prelll} as:
\begin{equation}\label{eq:prelll1}
\begin{split}
\int_{B_\rho}g(y)|(z-y)(y-z_0)|^{-1} dy  \le C|z-z_0|^{-1}\rho^{\frac{2}{r'}-1}.
\end{split}
\end{equation}
We can finally estimate \eqref{eqzya}:
\begin{equation}\label{eqzy1}
\begin{split}
\int_{B_\rho}g(y)|y-z_0|^{-1}|y|^{-k}|Dw(y)|dy &\le C\rho^{\frac{2}{r'}-1}\int_{\partial B_{\rho}} |z|^{-k}|z-z_0|^{-1}|Dw(z)|d|\gamma|(z) \\
&   + C\rho^{\frac{2}{r'}-1}\int_{B_{\rho}}|z|^{-k}|z-z_0|^{-1}g(z)|Dw(z)|dz.
\end{split}
\end{equation}
Observe that the second addendum is finite, due to the same bounds as in \eqref{boundw}. Now fix $\rho$ so small that
\begin{equation}\label{eqzy1200}
1 -C\rho^{\frac{2}{r'}-1} > 0
\end{equation}
 and absorb the second addendum of the right-hand side of \eqref{eqzy1} in the left-hand side to get
\begin{equation}\label{eqzy12}
\int_{B_\rho}g(z)|z-y|^{-1}|z|^{-k}|Dw(z)|dz \le C\rho^{\frac{2}{r'}-1}\int_{\partial B_{\rho}} |z|^{-k}|z-y|^{-1}|Dw(z)|d|\gamma|(z).
\end{equation}
Notice that we changed the name of the variable of integration from $y$ to $z$ in the first integral on the left-hand side, and we renamed $z_0$ as $y$. Going back to \eqref{eqzy}, this shows that $Dw = O(|y|^k)$.\ Hence, if $F_y(z) \doteq z^{-k}(z-y)^{-1}\Delta w_j(z)$, we can bound for some constants $c_1,c_2 > 0$:
\[
|F_y(z)| \overset{\eqref{eq:intineq}}{\le} c_1 |z|^{-k}|z-y|^{-1}g(z)|Dw|(z) \le c_2|z-y|^{-1}g(z),
\]
Since $g \in L^r(B_1)$ and $r > 2$, $F_y(\cdot)$ has equibounded $L^p(B_\rho)$ norms for some $p > 1$.\ In addition, $\lim_{y \to 0}F_y(z) = F_0(z)$ for a.e. $z \in B_\rho$.\ Thus, we can pass to the limit in \eqref{eqA} to show the assertion, namely:
\[
-2\pi i \lim_{y \to 0}y^{-k}\partial_zw_j(y) = \frac{i}{2}\int_{B_{\rho}}z^{-k-1}\Delta w_j(z)dz - \int_{\partial B_{\rho}} z^{-k-1}\partial_zw_j(z)d\gamma(z), \quad \forall j = 1,\dots, n.
\]
\end{proof}

\begin{prop}\label{selem}
Let $w$ fulfill the assumptions of Theorem \ref{thm:stronguq}. If $Dw(z) = o(|z-x|^k)$ for some $x \in B_1$ and for all $k \in \N$, then $Dw \equiv 0$ in $B_r(x)$ for some $r>0$.
\end{prop}
\begin{proof}
We again assume without loss of generality that $x = 0$, and moreover that $g \ge 1$ a.e in $B_1$. Due to our hypotheses, we see that \eqref{eqzy12} now holds for $y = 0$, any (fixed) $\rho > 0$ which fulfills \eqref{eqzy1200}, and all $k$:
\begin{equation}\label{eqzy000}
\int_{B_\rho}g(z)|z|^{-k-1}|Dw(z)|dz \le C\rho^{\frac{2}{r'}-1}\int_{\partial B_{\rho}} |z|^{-k-1}|Dw(z)|d|\gamma|(z) =  C\rho^{\frac{2}{r'}-2-k}\int_{\partial B_{\rho}}|Dw(z)|d|\gamma|(z).
\end{equation}
From the proof of Proposition \ref{filem} we know that the constant $C$ of \eqref{eqzy000} is independent of $k$. 
%Moreover, notice that inequality \eqref{eqzy000} actually holds for all $\rho < \rho_0$, where $\rho_0$ was chosen small in order to achieve.\ Choose any such $\rho$. 
Using that $g \ge 1$ a.e. and dividing both sides of \eqref{eqzy000} by $\rho^{-k-1}$, we get:
%\eqref{eqzy000} yields:
%\begin{equation}\label{eqzy123}
%\int_{B_\rho}|z|^{-k-1}|Dw(z)|dz \le C\rho^{\frac{2}{r'}-2-k}\int_{\partial B_{\rho}}|Dw(z)|d|\gamma|(z).
%\end{equation}
%Dividing both sides by $\rho^{-k-1}$, we obtain
\begin{equation}\label{eqzy1234}
	\int_{B_\rho}\frac{\rho^{k +1}}{|z|^{k+1}}|Dw(z)|dz \le C\rho^{\frac{2}{r'}-1}\int_{\partial B_{\rho}}|Dw(z)|d|\gamma|(z).
\end{equation}
Since $|z| \le \rho$ and the right-hand side does not depend on $k$, letting $k \to \infty$ shows $Dw\equiv 0$ on $B_\rho$ by Fatou's Lemma, as wanted.
%
%
%Now, either $Dw \equiv 0$ in the ball $B_{\rho_0/2}$ and the proof is concluded, or otherwise we can assume by contradiction that there exists $z_0 \in B_{\rho_0/2}$ such that $Dw(z_0) \neq 0$. Choose $\rho$ such that $|z_0| < \rho < 2|z_0|<\rho_0$. We can bound from below the left-hand side of \eqref{eqzy123} by
%\begin{equation}\label{eqzy223}
%c|z_0|^{-k-1} \le \int_{B_\rho}|z|^{-k-1}|Dw(z)|dz \le C\rho^{\frac{2}{r'}-2-k}\int_{\partial B_{\rho}}|Dw(z)|d|\gamma|(z),
%\end{equation}
%where $c = c(\rho,z_0)> 0$. Then \eqref{eqzy223} gives us:
%\[
%c \le  C\frac{\rho^{-1-k}}{|z_0|^{-1-k}},
%\]
%for all $k$ and for some $C = C(\rho)>0$ independent of $k$. This yields a contradiction by taking $k$ large enough since $|z_0| < \rho$. Thus $Dw \equiv 0$ in the ball $B_{\rho_0/2}$ as wanted.
\end{proof}

\begin{proof}[Proof of Theorem \ref{thm:stronguq}] Consider the closed set $$Z = \{x \in B_1: \forall r > 0 \text{ with }B_r(x) \subset B_1, B_r(x) \text{ contains at least two zeroes of }Dw\}.$$ Through Propositions \ref{filem}-\ref{selem} we see that $Z$ is open as well.\ Hence, either $Dw \equiv 0$ or the zeroes of $Dw$ are isolated. In the latter case, the claimed asymptotic expansion of $\partial_zw_j$ is then a consequence of Proposition \ref{filem}, and the fact that $a_{j_0} \neq 0$ for some $j_0$ is a consequence of Proposition \ref{selem}.
\end{proof}

\bibliographystyle{plain}
\bibliography{HMT-VerticalMeanCurvature-final}

\end{document}